\documentclass[12pt,twoside]{amsart}
\usepackage{amsfonts,amsmath}
\usepackage{enumerate}
\usepackage{mathtools}
\usepackage{xcolor,colortbl}
\usepackage{url}
\usepackage[hidelinks]{hyperref}
\usepackage{cleveref}
\usepackage{subcaption}

\def\endpf{\relax\ifmmode\expandafter\endproofmath\else
  \unskip\nobreak\hfil\penalty50\hskip.75em\hbox{}\nobreak\hfil\bull
  {\parfillskip=0pt \finalhyphendemerits=0 \bigbreak}\fi}
\def\bull{\vbox{\hrule\hbox{\vrule\kern3pt\vbox{\kern6pt}\kern3pt\vrule}\hrule}}

\newtheorem{defn}{Definition}[section]
\newtheorem{lemma}[defn]{Lemma}

\newtheorem{theorem}[defn]{Theorem}
\newtheorem{definition}[defn]{Definition}
\newtheorem{remark}[defn]{Remark}
\newtheorem{proposition}[defn]{Proposition}

\newtheorem{conjecture}[defn]{Conjecture}

\newtheorem{maintheorem}{Theorem}

\newcommand{\zz}{{\mathbb Z}}
\newcommand{\rr}{{\mathbb R}}

\newcommand{\qq}{{\mathbb Q}}

\newcommand{\spin}{\ifmmode{\rm Spin}\else{${\rm spin}$\ }\fi}
\newcommand{\spinc}{\ifmmode{{\rm Spin}^c}\else{${\rm spin}^c$\ }\fi}

\newcommand{\calu}{\mathcal{U}}

\newcommand{\CP}{\mathbb{CP}}

\definecolor{Gray}{gray}{0.8}

\newenvironment{narrow}[2]{%
 \begin{list}{}{%
  \setlength{\topsep}{0pt}%
  \setlength{\leftmargin}{#1}%
  \setlength{\rightmargin}{#2}%
  \setlength{\listparindent}{\parindent}%
  \setlength{\itemindent}{\parindent}%
  \setlength{\parsep}{\parskip}%
 }%
\item[]}{\end{list}}

\newif\ifpic
\picfalse    % no pictures
\DeclareMathOperator{\lk}{lk}
\DeclareMathOperator{\rk}{rk}
\DeclareMathOperator{\cross}{cr}

\begin{document}

\title{An algorithm to find \\ ribbon disks for alternating knots}
\author[Brendan Owens]{Brendan Owens}
\address{School of Mathematics and Statistics \newline\indent 
University of Glasgow \newline\indent 
Glasgow, G12 8SQ, United Kingdom}
\email{brendan.owens@glasgow.ac.uk}
\author[Frank Swenton]{Frank Swenton}
\address{Department of Mathematics \newline\indent 
Middlebury College \newline\indent
Middlebury, VT 05753, USA }
\email{fswenton@middlebury.edu}
\date{\today}
\thanks{B. Owens was supported in part by  EPSRC grant EP/I033754/1.}

\begin{abstract}  We describe an algorithm to find ribbon disks for alternating knots, and the results of a computer implementation of this algorithm.  The algorithm is underlain by a slice link obstruction coming from Donaldson's diagonalisation theorem.  It successfully finds ribbon disks for slice two-bridge knots and for the connected sum of any alternating knot with its reverse mirror, as well as for 662,903 prime alternating knots of 21 or fewer crossings.  We also identify some  examples of ribbon alternating knots for which the algorithm fails to find ribbon disks, though a related search identifies all such examples known.  Combining these searches with known obstructions, we resolve the sliceness of all but 3,276 of the over 1.2 billion prime alternating knots with 21 or fewer crossings.
\end{abstract}

\maketitle

\pagestyle{myheadings}
\markboth{BRENDAN OWENS AND FRANK SWENTON}{ALGORITHMIC RIBBON DISKS FOR ALTERNATING KNOTS}

%%%%%%%%%%%%%%%%%%%%%%%%%%%%%%%%%%%%%%%%%%%%%%%%%%%%%%%%%%%%%%%%%%%%%%%%%%

\section{Introduction}
\label{sec:intro}

A knot is a smooth simple closed curve in the 3-sphere, considered up to smooth isotopy.  A link is a disjoint union of one or more knots.  Knots and links are conveniently represented as diagrams, using projection onto a 2-sphere.  A knot (or link) is said to be \emph{alternating} if it admits an alternating diagram, in which one alternates between overcrossings and undercrossings as one traverses the knot.  A great deal of geometric information may be gleaned from an alternating diagram: for example one may use it to easily determine the knot genus \cite{crowell,murasugi58}, whether the knot is prime or admits mutants \cite{menasco}, and whether it has unknotting number one \cite{mccoy}.

A knot is said to be  \emph{slice} if it is the boundary of a smoothly properly embedded disk in the 4-ball.  Slice knots were introduced by Artin in the 1920s and have been studied extensively since the late 1950s.  Applications include the definition of the knot concordance group \cite{FM} and proofs of existence of exotic smooth structures on $\rr^4$ \cite{ras}.

The goal of this work is to find a method to determine from an alternating diagram whether a knot is slice, and in the affirmative case, to explicitly exhibit such a slice disk.  We will describe some partial progress towards this goal.  In particular we will describe a combinatorial algorithm that searches for slice disks for alternating knots.  A knot is called \emph{ribbon} if it bounds a slice disk to which the radial distance function on the 4-ball restricts to give a Morse function with no minima.  This is equivalent to the existence of a sequence of \emph{band moves}, $n$ in number, which convert the knot to the unlink of $(n+1)$ components.
Our algorithm finds such a disk by following a sequence of band moves and isotopies which are compatible with certain factorisations of Goeritz matrices associated to the alternating diagram.  We call these disks and the knots which admit such disks \emph{algorithmically ribbon}; as the following indicates, these turn out to be quite common.

\begin{maintheorem}
\label{thm:ribbonalg}
All slice two-bridge knots are algorithmically ribbon, as are all connected sums $-K\#K$ of an alternating knot with its mirror reverse and all but one slice prime alternating knot of 12 or fewer crossings.  There are $662,903$ algorithmically ribbon prime alternating knots of 21 or fewer crossings.
\end{maintheorem}

An example of an algorithmic ribbon disk for the stevedore knot ($6_1$ in the Rolfsen table \cite{rolfsen}) is shown in \Cref{fig:stevedoredisk}.  The figure shows a band move and a sequence of isotopies resulting in a 2-component unlink.  The decorations on the figure will be explained later.

\begin{figure}[htbp]
\centering
\includegraphics[scale=0.444444]{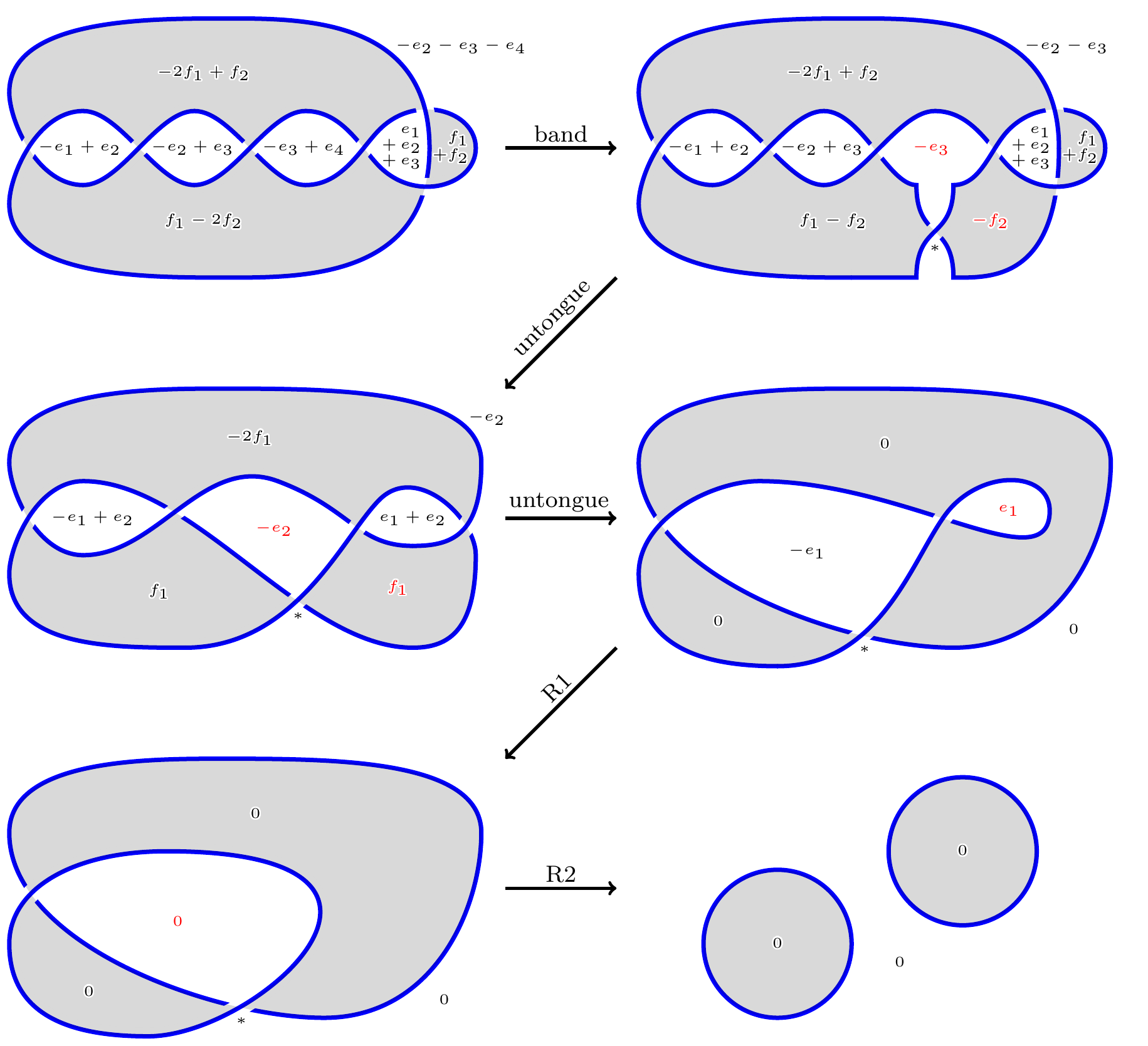}
\begin{narrow}{0.3in}{0.3in}
\caption
{{\bf The stevedore knot is algorithmically ribbon.} Crossings which make each diagram nonalternating are marked with an asterisk.}
\label{fig:stevedoredisk}
\end{narrow}
\end{figure}

The algorithm is based on a slice obstruction for links coming from Donaldson's diagonalisation theorem \cite{thmA}, which applies to links admitting certain special diagrams.  This is well-known for alternating knots and has been used to great success by Lisca in particular \cite{lisca1,lisca2}.  Recall that the nullity $\eta(L)$ of a link is equal to the first Betti number of its double branched cover, or equivalently to the nullity of the Goeritz matrix of a connected diagram of the link.  Nonsplit alternating links have nullity zero.  In general the nullity of a link is a lower bound for the number of nonalternating crossings in a connected diagram (\Cref{lem:semidef}).
We consider the larger class of \emph{minimally nonalternating links}, where a link $L$ is said to be minimally nonalternating if it admits a connected diagram with $\eta(L)$ nonalternating crossings.  A ribbon disk for an alternating knot $K$ is described by a sequence of band moves and isotopies from $K$ to an unlink.  Each of the band moves increases nullity by one, and both $K$ and the unlink are minimally nonalternating.  Our algorithm looks for such a sequence in which each intermediate diagram is minimally nonalternating and also \emph{bifactorizable}.  We say that a link diagram is bifactorizable if both of its Goeritz matrices admit a certain factorisation described in \Cref{sec:moves}.  bifactorizable diagrams are always minimally nonalternating.  Moreover, for minimally nonalternating diagrams, bifactorizability is equivalent to sliceness of the link being unobstructed by Donaldson's diagonalisation theorem (\Cref{prop:bif}).  Each bifactorizable diagram admits a finite set of band moves which preserve bifactorizability (\Cref{prop:algband}).  One of the key steps in our algorithm is to identify and test such band moves.  The other key step in the algorithm is to apply a sequence of simplifying isotopies to the diagram.  These isotopies, called \emph{generalised Tsukamoto moves}, also preserve bifactorizability.

The algorithm may also be applied to alternating links of more than one component, or more generally to minimally nonalternating links, to search for ribbon surfaces with prescribed Euler characteristic (\Cref{rem:algMNA}).

In \Cref{sec:moves} we describe the band moves and isotopies in more detail, and then in \Cref{sec:theory} we establish the underlying link sliceness obstruction and provide some further background information.  \Cref{sec:examples} contains proofs that all slice 2-bridge knots and all  knots of the form $-K\#K$, for alternating $K$, are algorithmically ribbon, as well as a conjecture which says that alternating knots which admit ribbon disks with a single saddle point are in fact algorithmically ribbon. \Cref{sec:imp} describes the implementation of the algorithm, which is built into the software package KLO of the second author.  This software, together with search results and instructions for using the algorithm on examples, is available from \url{www.klo-software.net/ribbondisks}.  \Cref{sec:results} describes the results of the algorithm when applied to alternating knots of small crossing number.  \Cref{sec:escapees} describes a search for \emph{escapees}, which are ribbon alternating knots which are not algorithmically ribbon, though they really want to be; 
combining this with obstructions due to Fox-Milnor and Herald-Kirk-Livingston, the result is a resolution of the question of sliceness for all but 3,276 of the over 1.2 billion prime alternating knots of at most 21 crossings, including all such knots of crossing number at most 15.

\vskip2mm
\noindent{\bf Acknowledgements.}  
The first author has benefitted from helpful conversations on this topic with many colleagues over the years, including Josh Greene, Paolo Lisca, and Sa\v{s}o Strle.  Particular thanks are due to Paolo Lisca and Liam Watson for discussions on minimally nonalternating links, and to Tetsuya Tsukamoto who showed us the untongue-2 move.  We thank Cornelia Van Cott for a helpful question on an earlier version of the paper, and we thank the anonymous referees for many helpful comments to improve the exposition.

%%%%%%%%%%%%%%%%%%%%%%%%%%%%%%%%%%%%%%%%%%%%%%%%%%%%%%%%%%%%%%%%%%%%%%%%%%

\section{Algorithmic band moves and isotopies}
\label{sec:moves}

In this section we describe some special embedded surfaces in the 4-ball called \emph{algorithmic ribbon surfaces}, which will be realised as sequences of moves applied to link diagrams.

Following Tait \cite{tait} and Goeritz \cite{goeritz}, we make use of chessboard colourings of knot and link diagrams.  As usual, a link diagram is a 4-valent graph embedded in $S^2$ with over- and under-crossings indicated at the vertices, which are then referred to as crossings.  A chessboard colouring of a link diagram is a colouring of the regions in the complement of the graph using two colours, white and black, with the regions on either side of any edge having different colours.  A  chessboard-coloured diagram of $L$ gives rise to a spanning surface $F_b$, called the black surface: this is obtained by gluing the black regions of the diagram, with a half-twisted band at each crossing.
For alternating diagrams we follow the colouring convention shown in \Cref{fig:col1}.  In a general diagram we refer to crossings coloured as in \Cref{fig:col1} as \emph{alternating} crossings, and to crossings coloured as in \Cref{fig:col2} as \emph{nonalternating} crossings, which we label with asterisks.  We define an alternating diagram to be one that admits a chessboard colouring without nonalternating crossings; Greene and Howie have recently shown that the resulting black and white surfaces give rise to a topological characterisation of nonsplit alternating links \cite{greene,howie}.

\begin{figure}[htbp]
\centering
\begin{subfigure}{2in}
\centering
\includegraphics[scale=1]{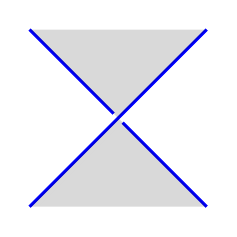}
\vspace{-0.2in}
\caption{}
\label{fig:col1}
\end{subfigure}
\hspace{0.2\textwidth}
\begin{subfigure}{2in}
\centering
\includegraphics[scale=1]{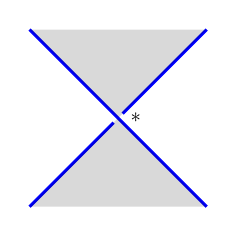}
\vspace{-0.2in}
\caption{}
\label{fig:col2}
\end{subfigure}
\begin{narrow}{0.3in}{0.3in}
\caption
{{\bf Colouring conventions.} An alternating and a nonalternating crossing.}
\label{fig:crossings}
\end{narrow}
\end{figure}

A  chessboard-coloured link diagram gives rise to two Goeritz matrices as we now recall, cf. \cite{goeritz,GL}.  Choose a labelling $R_0,\dots,R_m$ of the white regions of the diagram.  To each crossing $c$ in the diagram we associate a number $\epsilon(c)$ which is $+1$ if $c$ is an alternating crossing and $-1$ if $c$ is a nonalternating crossing.  We form a matrix $\widehat{G}$ with $m+1$ rows and columns whose entries are as follows:
$$g_{ij}=
\begin{cases}
-\sum\epsilon(c), \mbox{ summed over crossings incident to $R_i$ and $R_j$,} & \mbox{ if $i\ne j$}\\
-\sum_{k\ne i}g_{ik}, & \mbox{ if $i=j$}.
\end{cases}$$
The \emph{black Goeritz matrix} $G_b$ of the diagram is then defined to be the matrix obtained by omitting the $0$th row and $0$th column of $\widehat{G}$.

A \emph{spanning surface} for a link $L$ in $S^3$ is an embedded surface $F$, without closed components, bounding $L$ in $S^3$.  An example is given by the black surface of a chessboard-coloured diagram.  The Gordon-Litherland pairing on $H_1(F;\zz)$ is given by
$$\lambda([a],[b])=\lk(a,\tau b),$$
where the linking number is taken between $a$ and the double normal push-off $\tau b$ of $b$.  This agrees with the symmetrised Seifert pairing if $F$ happens to be orientable.  We have the following result of Gordon and Litherland \cite{GL}.

\begin{theorem}
\label{thm:GL}
The lattice $(H_1(F;\zz),\lambda)$ is isomorphic to the intersection lattice of the 4-manifold $X_F$ given as the double cover of $D^4$ branched along a copy of the surface $F$ with its interior pushed inside the 4-ball.
\end{theorem}

In general a lattice $\Lambda$ is a pair consisting of a finitely generated free abelian group together with a symmetric integer-valued bilinear pairing.  The rank of the lattice is the rank of the underlying group.

A basis for the first homology of a connected $F_b$ is given by curves, each of which winds once around one of the white regions $R_1$,\dots,$R_m$ as above.  When $F_b$ is disconnected, one or more of the white regions $R_i$ may have more than one boundary component, in which case the corresponding basis element is a multicurve in $F_b$ with one curve for each boundary component of $R_i$, oriented as the boundary of $R_i$.  The black Goeritz matrix $G_b$ is the matrix of the Gordon-Litherland form with respect to this basis.  The lattice $\Lambda_b:=(H_1(F_b;\zz),\lambda)$ is called the black Goeritz lattice, or simply the black lattice, of the diagram.

We define the white Goeritz matrix $G_w$ and the white lattice $\Lambda_w$ of the diagram to be equal to the black Goeritz matrix and lattice of the mirror diagram, obtained by reversing every crossing in $D$, with black and white switched.  By \Cref{thm:GL}, $\Lambda_b$ is the intersection lattice of the double cover of the 4-ball branched along the pushed-in black surface of $D$.  Similarly $\Lambda_w$ is the intersection lattice of the double cover of the 4-ball branched along the pushed-in white surface of $D$, with orientation reversed.  Note that we have chosen our conventions so that both Goeritz lattices of an alternating diagram are positive definite.

Recall that a $\mu$-component link is a smooth embedding of the disjoint union of $\mu$ circles into $S^3$.  The \emph{diagram components} of a diagram of a link are the connected components in $S^2$ of the given projection; thus the number of diagram components is less than or equal to the number of components of the link.

Given a lattice $\Lambda=(H,\lambda)$ we let $\Lambda^0$ denote its nullspace; this is the subspace consisting of elements $a\in H$ with $\lambda(a,b)=0$ for all $b\in H$.
For a general chessboard-coloured diagram, the Goeritz lattices may have nontrivial nullspaces.  Quotienting each Goeritz lattice by its nullspace gives rise to a lattice with a nondegenerate pairing.  Let $D$ be a chessboard-coloured diagram with $m+1$ white regions and $n+1$ black regions.  We say that $D$ is \emph{bifactorizable} if each of these nondegenerate Goeritz quotients embeds as a finite-index sublattice of a standard diagonal lattice.  That is to say, $D$ is bifactorizable if there exist lattice embeddings
\begin{equation}
\label{eq:latemb}
\begin{aligned}
\Lambda_b/\Lambda_b^0&\rightarrow\zz^{m'},\quad\mbox{and}\\
\Lambda_w/\Lambda_w^0&\rightarrow\zz^{n'},
\end{aligned}
\end{equation}
where $m'=\rk(\Lambda_b/\Lambda_b^0)=m-\rk\Lambda^0_b$ and  $n'=\rk(\Lambda_w/\Lambda_w^0)=n-\rk\Lambda^0_w$.   
We will see in \Cref{lem:latdim} that  these ranks can be read off from the diagram as follows:

\begin{equation*}
\label{eq:bifacranks}
\begin{aligned}
m'&=m-k-l_w+1,\quad\mbox{and}\\
n'&=n-k-l_b +1,
\end{aligned}
\end{equation*}
where $k$ is the number of nonalternating crossings in $D$ and $l_b$ (respectively, $l_w$) is the number of components of the black (resp., white) surface.
In particular, for a connected alternating bifactorizable diagram, $m=m'$ and $n=n'$.

Such embeddings as in \eqref{eq:latemb} are equivalent to a pair of lattice morphisms
\begin{equation}
\label{eq:latmorph}
\begin{aligned}
\Lambda_b&\rightarrow\zz^{m'},\quad\mbox{and}\\
\Lambda_w&\rightarrow\zz^{n'},
\end{aligned}
\end{equation}
with finite-index images.

Choosing a basis of regions as above for each of the Goeritz lattices, and an orthonormal basis for each standard diagonal lattice, we see that bifactorizability is equivalent to the existence of 
matrices $A\in M(m'\times m,\zz)$ and $B\in M(n'\times n,\zz)$ satisfying
\begin{equation*}
\begin{aligned}
G_b&=A^TA,\\
G_w&=B^TB,
\end{aligned}
\label{eqn:bif}
\end{equation*}
the existence of which can be determined algorithmically, see for example \cite{plesken}.
We represent a bifactorization on a diagram as follows.   In each white region, representing a generator of $\Lambda_b$,  we write the corresponding element of $\zz^{m'}$, in terms of the standard basis $e_1,\dots,e_{m'}$; the coefficients are the entries of the corresponding column of $A$.  We similarly write an element of $\zz^{n'}$, in terms of the basis $f_1,\dots,f_{n'}$, in each black region, using the entries of $B$; for an example, see \Cref{fig:stevedoredisk}.  Note that the lattice morphism is easily extended to the region $R_0$, which was excluded when calculating the Goeritz matrix, by noting that the lattice elements corresponding to the white regions $R_0,\dots,R_m$ sum to zero.

A \emph{band move} on a link $L$ in $S^3$ is represented by a smoothly embedded rectangle $\beta$ in $S^3$ with opposite ends attached to $L$ and the rest of the rectangle embedded in the complement of $L$.  The result of the band move is the link $L'$ which is obtained by removing the ends of the rectangle $\beta$ from $L$ and adding the other two sides, followed by smoothing.
Suppose that we are given a diagram $D$ for a link $L$ and a band move, in the form of a rectangle $\beta$, which we wish to apply to $L$.  We may assume that $\beta$ is put in \emph{standard position} on $D$, by which we mean that it is represented by an arc with endpoints on the diagram,  intersecting the diagram transversely in a number of crossings, and with no self-crossings, together with an integer on each component of the intersection of the arc with the complement of the diagram.  The arc is the core of the band.  The integers count the number of half-twists relative to the blackboard framing, and their sum is called the twist number of the band.  Every band may be put in standard position as the reader may verify; in particular, self-crossings may be successively replaced by pairs of crossings between the core of the band and $D$ by isotopy.  The length of a band $\beta$ in standard position is the number of regions it crosses, counted with multiplicity, or equivalently one fewer than the number of intersections between the core and $D$.  The first move in \Cref{fig:stevedoredisk} is a band move whose band has length and twist number equal to one.  The band shown on the left in \Cref{fig:algbands} has length one and twist number $-1$, and that on the right has length two and twist number zero.

\begin{figure}[htbp]
\centering
\includegraphics{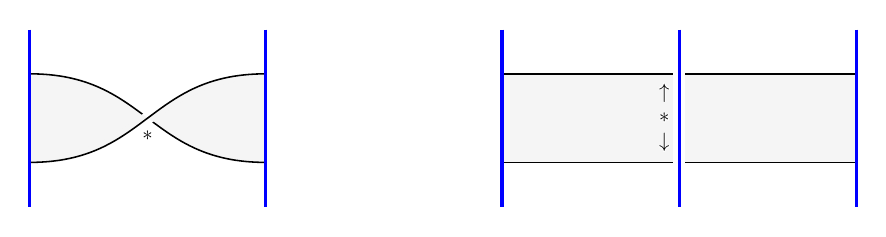}
\begin{narrow}{0.3in}{0.3in}
\caption
{{\bf Candidate algorithmic bands.} For the length-one band, the twist is chosen such that the new crossing is non-alternating; for the length-two band, one of the two crossings will be non-alternating, depending on the chessboard colouring of the diagram and whether the band crosses over or under.}
\label{fig:algbands}
\end{narrow}
\end{figure}

The following definitions describe some special band moves and isotopies which we will search for in order to find ribbon disks for (some but not all) slice alternating knots.

\begin{defn}
\label{def:algband}
Let $D$ be a bifactorizable diagram.  An \emph{algorithmic band} for $D$ is a band $\beta$ which is either 
\begin{enumerate}
\item of length 2 with twist number 0, or
\item of length 1 with twist number $\pm1$, with sign chosen so that the crossing introduced is nonalternating,
\end{enumerate}
and for which the diagram obtained from $D$ by applying the band move determined by $\beta$ is bifactorizable.
\end{defn}

Candidate algorithmic bands are shown in \Cref{fig:algbands}.  It is important to note that there are finitely many such bands in any diagram, and that most such bands in a given bifactorizable diagram will not satisfy the bifactorizability condition.

In this paper we take a tangle diagram to mean part of a diagram which is contained in a disk.  A tangle diagram replacement is then a modification of a diagram preserving the part of the diagram outside that disk.

\begin{definition}
\label{def:gTsuk}
A {\em generalised Tsukamoto move} is a tangle diagram replacement $\Gamma_1\rightarrow\Gamma_2$ satisfying the following properties:
\begin{enumerate}
\item isotopy: $\Gamma_2$ is isotopic to $\Gamma_1$ relative to boundary;
\item crossings: $\Gamma_2$ has at most as many crossings as $\Gamma_1$, and the same sum of number of components plus number of nonalternating crossings as $\Gamma_1$;
\item bifactorizability is preserved: any bifactorization of a diagram containing $\Gamma_1$ gives rise to one for the diagram with $\Gamma_1$ replaced by $\Gamma_2$, and vice versa.
\end{enumerate}
\end{definition}

Figures \ref{fig:flype} through \ref{fig:tripleflype} contain examples of generalised Tsukamoto moves.  In each case one may obtain further moves by reflecting the diagram and switching black and white.  
Other versions of the 3-flype move in Figure \ref{fig:tripleflype} may be obtained by replacing the diagram on the right as follows: first simplify the diagram on the left in the obvious way by undoing the half twist in three strands on either side of the tangle, and  turning the tangle over.  Then place a ``floating loop", as shown in Figure \ref{fig:loopthreeflype}, around the tangle, either entirely over or entirely under the rest of the diagram, and resolve one of the three resulting nonalternating crossings.  One can then obtain another version from each of these by reflecting the diagram and switching black and white.
We sometimes refer to generalised Tsukamoto moves as AA-simplifications, where AA stands for ``almost alternating".  We are grateful to Tatsuya Tsukamoto who suggested the move shown in \Cref{fig:untongue2}.  The moves in Figures \ref{fig:flype}, \ref{fig:AAreidemeister}, and \ref{fig:untongue} were shown by Tsukamoto in \cite{tsuk2} to be sufficient to simplify any almost-alternating diagram of the two-component unlink (see \Cref{thm:tsuk} below).

We briefly justify the preservation of bifactorizability in the case of one colour for the flype move in Figure \ref{fig:flype}; the arguments are similar for the other colour and the other moves.  

\begin{lemma}
\label{lem:flype}
If the labels on the black regions of the left hand diagram of Figure \ref{fig:flype} extend to an embedding in $\zz^n$ of the white lattice of a diagram, then so do those in the right hand diagram.
\end{lemma}
\begin{proof}
Let $R_i$ denote the black regions labelled by vectors $w_i$ for $i\in\{1,2,3\}$ in the left hand diagram of the figure.
We assume that the regions $R_1$ and $R_2$ are distinct (if not the diagram is to be interpreted as associating the vector $w_1+w_2$ to the common region). First of all, any black region $R'$ not shown in the diagram has some vector $w'$ associated to it that remains unchanged for the diagram on the right.  Since $R'$ does not share any crossings with $R_2$ or any region in the tangle, it follows that 
\begin{align*}
w'\cdot w_1&=w'\cdot (w_1+w_2+\sum w),\\
w'\cdot w_3&=w'\cdot (w_2+w_3),\\
0&=w'\cdot (-w_2-\sum w),
\end{align*}
where $\sum w$ denotes the sum of vectors associated to all black regions in the tangle box.  We see that this correctly matches the Goeritz lattice pairings for the region $R'$.

Similarly for any black region $R$ in the tangle box on the left hand diagram, we let $w$ be the vector associated to it by the Goeritz lattice embedding.  Recalling that the sum of all black regions is zero in the Goeritz lattice, and noting that the only black regions which may share crossings with $R$ are the other regions in the tangle box together with $R_1$ and $R_2$, we conclude that 
$$w\cdot w_3=w\cdot (w_1+w_2+\sum w)=0.$$

From this we see that $w\cdot w_1=w\cdot(-w_2-\sum w)$ and $w\cdot w_2=w\cdot(w_2+w_3)$.  Since the number of crossings between $R$ and any other black region in the tangle box is unchanged by the flype, we conclude that the labels on the right hand diagram give the correct lattice pairings for the region $R$.

Finally we note that the region $R_2$ has no crossings with any region outside of the portion of the diagram shown, from which it follows that
$$w_2\cdot (w_1+w_2+w_3+\sum w)=0$$ and
hence
$$1=w_2\cdot w_3=-w_2\cdot(w_1+w_2+\sum w)=(-w_2-\sum w)\cdot(w_1+w_2+\sum w).$$

We conclude that the vector labels on the right hand diagram represent an embedding of its Goeritz lattice as required.
\end{proof}

We recall that a surface $F$ in the 4-ball may be perturbed so that the radial distance function $\rho$ restricts to give a Morse function on $F$.  This gives rise to an embedded handle decomposition or movie presentation of $F$, in which minima appear as split disjoint unknots, maxima manifest as disappearing split unknots, and saddles appear as band moves.  An isotopy of the link may be interpreted as an embedded cylinder between two levels of $\rho$.

\begin{defn} 
\label{def:algsurf}
Let $L$ be a nonsplit alternating link.  An \emph{algorithmic ribbon surface} for $L$ is a surface embedded in the 4-ball which is represented by a sequence of algorithmic band moves and generalised Tsukamoto moves, starting with a nonsplit alternating diagram of $L$ and ending with a crossingless unlink diagram.
If such a surface exists, we say that $L$ is \emph{algorithmically ribbon}.
\end{defn}

We note that algorithmic ribbon surfaces have Euler characteristic one: each algorithmic band move introduces one nonalternating crossing and preserves the number of diagram components, and each generalised Tsukamoto move preserves the sum of the number of nonalternating crossings plus the number of diagram components.  It follows that the number of components of the unlink at the end of the sequence in \Cref{def:algsurf} is equal to one plus the number of band moves in the sequence, so that the surface has Euler characteristic one.  If the starting point of the sequence is a knot diagram, then the sequence represents a disk, which we call an algorithmic ribbon disk.  Note also that given a sequence of candidate algorithmic bands and generalised Tsukamoto moves, starting with a connected alternating diagram $D$ and ending with a crossingless unlink diagram, one can work backwards from the crossingless diagram to uniquely obtain a bifactorization of each diagram in the sequence, so that the sequence in fact represents an algorithmic ribbon surface.  One could thus search for such surfaces without checking the bifactorizability condition, but in practice such a search would be very slow due to the large number of candidate algorithmic bands to consider.

For nonsplit alternating links with more than one component, algorithmically ribbon does not imply slice; rather it implies that the link bounds a surface of Euler characteristic one, not necessarily orientable, with no closed components.  Such links were called $\chi$-slice in \cite{lconc}.

\begin{remark}
\label{rem:algMNA}
The algorithm may be applied more generally to \emph{minimally nonalternating links}, as defined in the next section.  An algorithmic surface for a minimally nonalternating link $L$ of nullity $\eta$ is a ribbon surface of Euler characteristic $\eta+1$.  If the number of components of $L$ is equal to $\eta+1$, then an algorithmic surface for $L$ consists of $\eta+1$ disjoint disks.
\end{remark}

\begin{remark}
\label{rem:orble?}
Note that an algorithmic ribbon surface does not have to be orientable in general, and thus in particular algorithmic bands do not respect link orientation.  In the case of an alternating knot, however, an algorithmic surface is in fact a disk, and so \emph{a posteriori} we deduce that all of its algorithmic bands do in fact respect an orientation of the original knot.   More generally one may observe, for example using Turaev's theorem relating spin structures on the double branched cover of a link and orientations \cite{turaev}, that the nullity of a link is bounded above by the number of components minus one.  It follows from this, or from the previous remark, that if a minimally nonalternating link has nullity equal to the number of its components minus one, then each successive algorithmic band must in fact increase the number of components, and is thus compatible with any orientation of the link.
\end{remark}

\begin{figure}[htbp]
\centering
\includegraphics[scale=0.6]{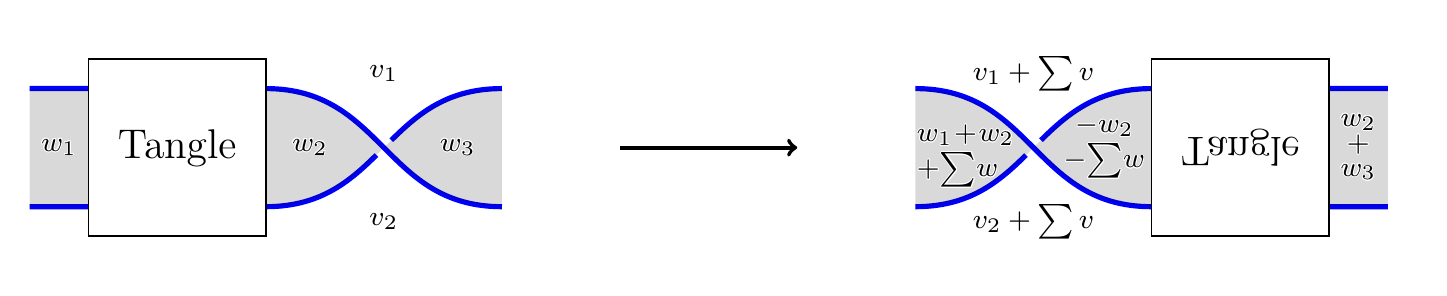}
\begin{narrow}{0.3in}{0.3in}
\caption
{{\bf Flype.} Here $\sum v$ (resp., $\sum w$) is a sum over all the white (resp., black) regions in the tangle.  Vectors associated to white regions in the tangle are multiplied by $-1$ in this move; vectors associated to black regions in the tangle are unchanged.  Note also that a special case of the flype is to move a connected summand past a crossing.}
\label{fig:flype}
\end{narrow}
\end{figure}

\begin{figure}[htbp]
\centering
\includegraphics[scale=0.75]{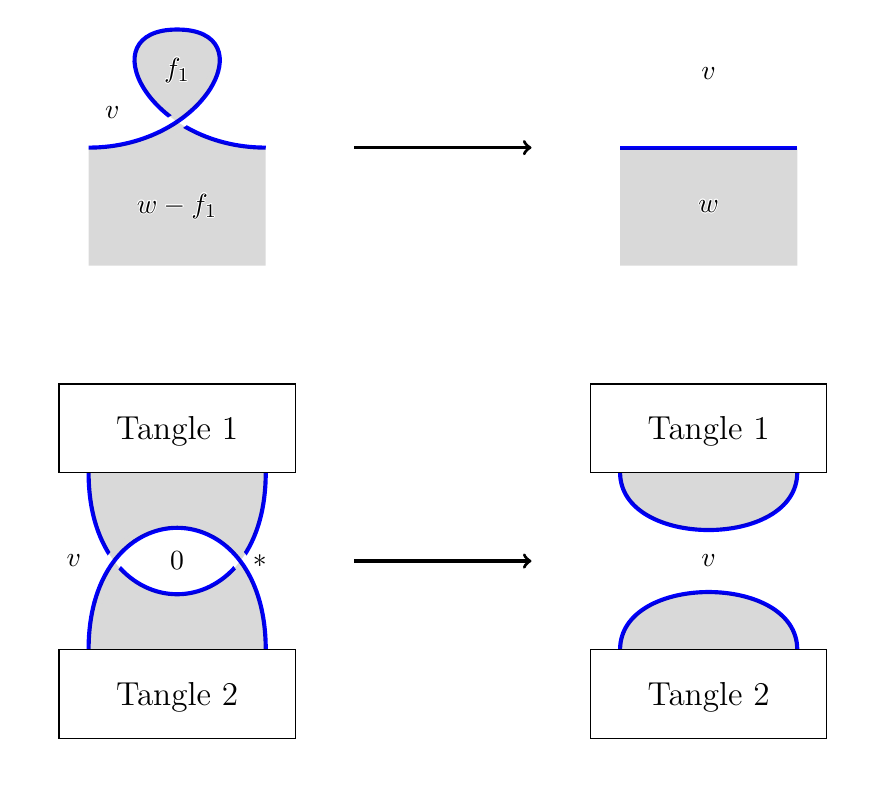}
\begin{narrow}{0.3in}{0.3in}
\caption
{{\bf Reidemeister moves preserving bifactorizability.} Note that both moves preserve the sum of number of diagram components plus number of nonalternating crossings. Note also that Reidemeister 3 moves change the number of nonalternating crossings and do not preserve bifactorizability.}
\label{fig:AAreidemeister}
\end{narrow}
\end{figure}

\begin{figure}[htbp]
\centering
\includegraphics[scale=0.6]{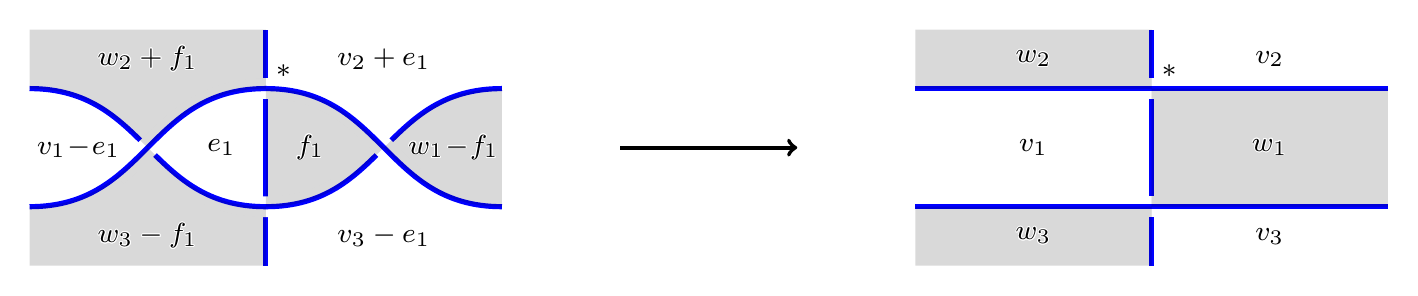}
\begin{narrow}{0.3in}{0.3in}
\caption
{{\bf Untongue move.} }
\label{fig:untongue}
\end{narrow}
\end{figure}

\begin{figure}[htbp]
\centering
\includegraphics[scale=0.6]{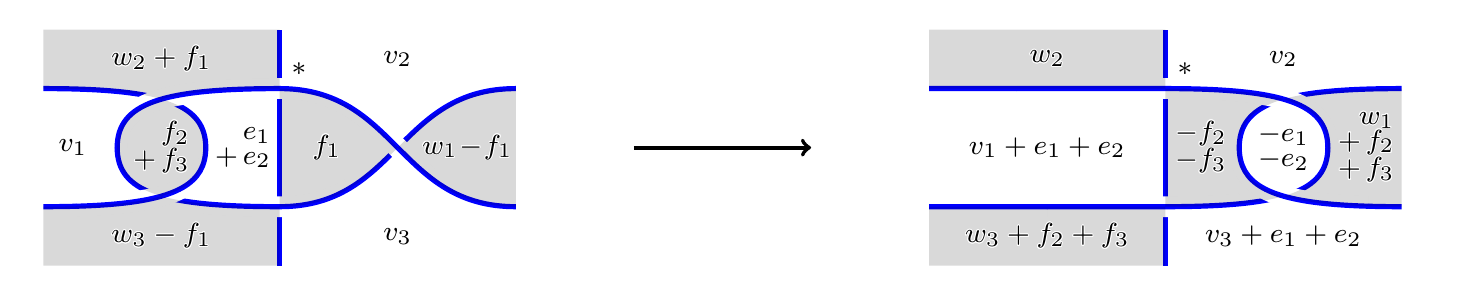}
\begin{narrow}{0.3in}{0.3in}
\caption
{{\bf Almost-alternating clasp move.} }
\label{fig:clasp}
\end{narrow}
\end{figure}

\begin{figure}[htbp]
\centering
\includegraphics[scale=0.6]{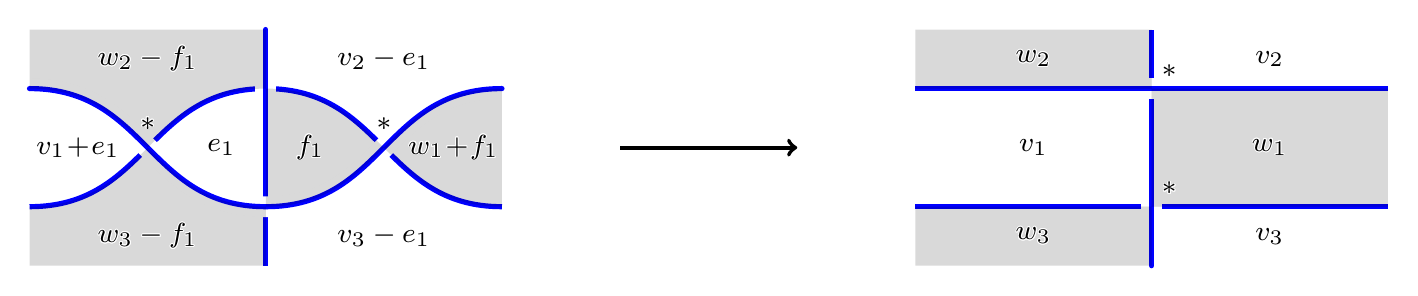}
\begin{narrow}{0.3in}{0.3in}
\caption
{{\bf Untongue-2 move.} }
\label{fig:untongue2}
\end{narrow}
\end{figure}

\begin{figure}[htbp]
\centering
\includegraphics[scale=0.5]{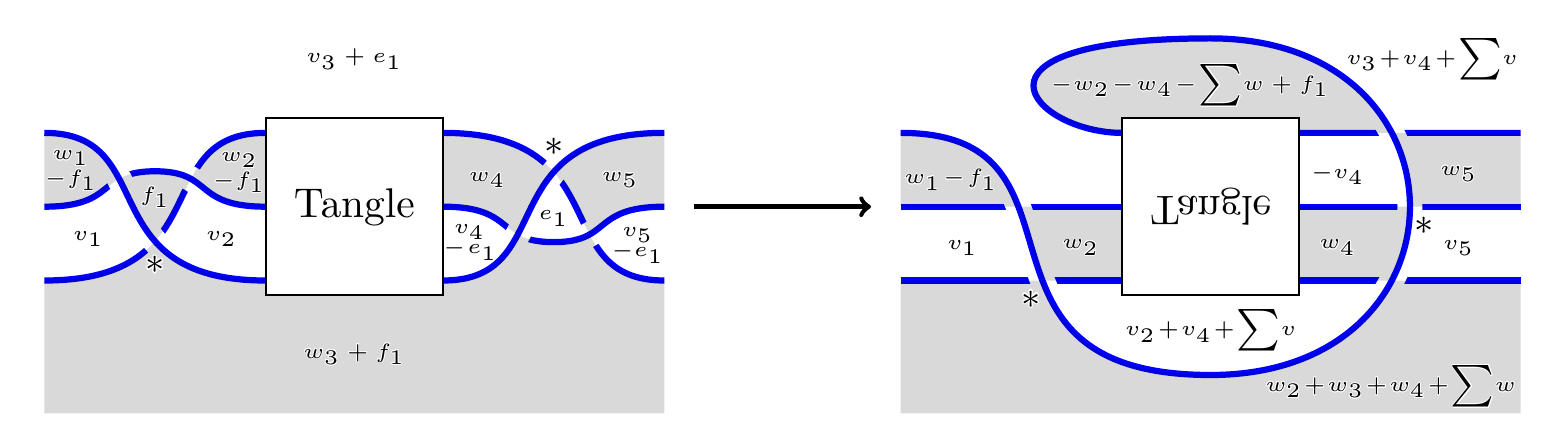}
\begin{narrow}{0.3in}{0.3in}
\caption
{{\bf The 3-flype move.} Here $\sum v$ (resp., $\sum w$) is a sum over all the white (resp., black) regions in the tangle.    Vectors associated to white regions in the tangle are multiplied by $-1$ in this version of the move; black regions in the tangle are unchanged.}
\label{fig:tripleflype}
\end{narrow}
\end{figure}

\clearpage

%%%%%%%%%%%%%%%%%%%%%%%%%%%%%%%%%%%%%%%%%%%%%%%%%%%%%%%%%%%%%%%%%%%%%%%%%%

\section{Obstructions and conjectures}
\label{sec:theory}

Recall that the nullity $\eta(L)$ of a link in $S^3$ is the rank of the nullspace of a Goeritz matrix of a connected diagram of $L$.  Equivalently, it is the first Betti number of the double cover of $S^3$ branched along $L$, noting that the Goeritz matrix is the intersection form of a simply-connected 4-manifold bounded by the double branched cover.  
The following lemma gives a generalisation of the fact that, with our conventions from \Cref{sec:moves}, both Goeritz matrices of a nonsplit alternating diagram are positive definite.

\begin{lemma}
\label{lem:semidef}
Let $L$ be a link in $S^3$, and let $D$ be a chessboard-coloured diagram of $L$ with $l$ diagram components and $k$ nonalternating crossings.  Then the nullity of $L$ is bounded above by $k+l-1$,
\begin{equation}
\label{eq:etakl}
\eta(L)\le k+l-1.
\end{equation}
Moreover, equality is attained in \eqref{eq:etakl} if and only if both Goeritz matrices of $D$ are positive semi-definite.  In particular, if $D$ is bifactorizable then equality is attained in \eqref{eq:etakl}.
\end{lemma}
\begin{proof}  We will give a topological proof involving double branched covers; a proof along the lines of \cite[Proposition 4.1]{greene} should also be possible.  Given any chessboard-coloured diagram $D$ of a link $L$, one may form the smooth closed oriented manifold $X=\Sigma_2(S^4,F_b\cup F_w)$ that is the double cover of $S^4$ along the union of the black and white surfaces, with the interior of the black surface pushed into one hemisphere and the white into the other.  We fix the orientation by specifying that the 3-sphere is the oriented boundary of the hemisphere into which the black surface is pushed.  By construction, this manifold splits along the double branched cover of $L$ into two pieces whose intersection lattices are the Goeritz lattices of $D$.  In other words,
$$X=X_1\cup X_2,$$
where $X_1$ is the double cover of the 4-ball branched along the pushed-in black surface of $D$, and $X_2$ is the double-branched cover of the white surface, with orientation reversed.

Consider first a one-crossing diagram of the unknot.  Choose a chessboard colouring so that the black surface is a M\"{o}bius band and the white surface is a disk.  Then $\Sigma_2(S^4,F_b\cup F_w)$ is $\CP^2$ if the crossing is alternating and $\overline{\CP^2}$ if the crossing is nonalternating (see for example \cite{ak}).

Now consider a general diagram $D$ as in the statement.  If the diagram is disconnected, then one may apply a Reidemeister 2 move (reversing that shown in \Cref{fig:AAreidemeister}).  This reduces the number of diagram components by one and increases the number of nonalternating crossings by one.  It also leaves one Goeritz matrix unchanged and adds a zero row and column to the other, preserving semi-definiteness for both.  Thus it suffices to prove the lemma for a connected diagram.
Assume therefore that  $D$ is a connected diagram with $k$ nonalternating crossings, and we wish to show that $$\eta(L)\le k.$$
Removing a small ball around each crossing gives a diagram in a punctured disk whose black and white surfaces together form a punctured unknotted sphere; the double branched cover is then a punctured $S^4$.  Filling in each puncture corresponds to taking connected sum with the double branched cover arising from a one-crossing diagram.  Thus for a  connected diagram with $j$ alternating crossings and $k$ nonalternating crossings,
$$X=X_1\cup X_2=\Sigma_2(S^4,F_b\cup F_w)\cong j\CP^2\# k\overline{\CP^2}.$$
From the Mayer-Vietoris sequence together with the fact that, since $D$ is connected, $X_1$ and $X_2$ admit handlebody decompositions consisting of one 0-handle and some number of 2-handles \cite{a2016,ak}, we see that
the second Betti number of $X$ is the sum of those of $X_1$ and $X_2$; also by Novikov additivity, the signature of $X$ is the sum of those of $X_1$ and $X_2$.  Moreover from the long exact sequence of the pair $(X_i,\partial X_i)$, we see that, for a connected diagram, the rank of the nullspace of $H_2(X_i)$ is $\eta(L)$.  This yields
\begin{align*}
j+k&=b_2^+(X_1)+b_2^-(X_1)+b_2^+(X_2)+b_2^-(X_2)+2\eta(L)\quad\text{and}\\
j-k&=b_2^+(X_1)-b_2^-(X_1)+b_2^+(X_2)-b_2^-(X_2),
\end{align*}
from which we find
$$k=b_2^-(X_1)+b_2^-(X_2)+\eta(L).$$
This gives $\eta(L)\le k$, with equality if and only if both $X_1$ and $X_2$ are positive semi-definite.

For the last sentence of the statement, observe that if a diagram is bifactorizable, with lattice morphisms as in \eqref{eq:latmorph}, then it follows that both Goeritz lattices $\Lambda_b$ and $\Lambda_w$ are positive semi-definite.
\end{proof}

We say that a link diagram is \emph{minimally nonalternating} if it attains equality in \eqref{eq:etakl}.  By \Cref{lem:semidef}, this is equivalent to both Goeritz matrices of the diagram being positive semi-definite, with our conventions from \Cref{sec:moves}.  We say a link is minimally nonalternating if it admits a minimally nonalternating diagram.  Minimally nonalternating links have been studied in unpublished work of the first author with P.~Lisca and L.~Watson; they may be seen as a natural generalisation of alternating links.

\begin{remark}
\label{rem:MNA}
It would be interesting to know if the characterisation of nonsplit alternating links due to Greene and Howie \cite{greene,howie} extends to links with diagrams attaining equality in \eqref{eq:etakl}.  In other words, is existence of a minimally nonalternating diagram for $L$ equivalent to existence of positive semi-definite surfaces in $S^3$ for both $L$ and its mirror?
\end{remark}

We now confirm the rank formulae given in \Cref{sec:moves} for a bifactorizable diagram.

\begin{lemma}
\label{lem:latdim}
Let $D$ be a bifactorizable link diagram, with $m+1$ white regions, $n+1$ black regions, and $k$ nonalternating crossings.  Let $l_b$ be the number of components of the black surface and let $l_w$ be the number of components of the white surface.  The rank of the quotient  of the black Goeritz lattice by its nullspace is 
$$\rk(\Lambda_b/\Lambda^0_b)=m-k-l_w+1,$$
and similarly
$$\rk(\Lambda_w/\Lambda^0_w)=m-k-l_b+1.$$
\end{lemma}
\begin{proof}
As noted above, bifactorizability implies equality in \eqref{eq:etakl} by \Cref{lem:semidef}.  Also the nullity $\eta$ of $L$ is equal to the nullity of either Goeritz form of any connected diagram.  Note that the R2 move shown in Figure \ref{fig:AAreidemeister} changes the rank of $\Lambda^0_b$ by one but does not change $\Lambda_w$.  By consideration of shading in each of a minimal sequence of R2 moves to transform $D$ to a connected diagram, we find that
$$\eta=\rk\Lambda^0_b+l_b-1,$$
and also that the number of components of $D$ is $l=l_b+l_w-1$.  Combining these identities with the equality in \eqref{eq:etakl} yields the formula for the rank of $\Lambda_b/\Lambda^0_b$, and the same argument applies to the white lattice.
\end{proof}

\begin{proposition}
\label{prop:slicelike}
Let $F$ be a smoothly properly embedded surface in the 4-ball with no closed components and bounded by a link $L$ in $S^3$.  Then the Euler characteristic of $F$ is bounded above by one plus the nullity of $L$,
\begin{equation}
\label{eq:chieta}
\chi(F)\le\eta(L)+1.
\end{equation}
If equality is attained in \eqref{eq:chieta} then the double  cover $\Sigma_2(D^4,F)$ of the 4-ball branched along $F$ is a rational homology $\displaystyle\natural_{\eta(L)}(S^1\times D^3)$.
\end{proposition}

\begin{proof}
This follows easily from the proof of \cite[Proposition 3.1]{aceto}.
\end{proof}

A $\mu$-component link is said to be \emph{strongly slice} if it bounds $\mu$ disjoint smoothly properly embedded disks in the 4-ball.  Strongly slice links with their slicing disks attain equality in \eqref{eq:chieta}.  We note that the following proposition gives an obstruction for a minimally nonalternating link to be strongly slice.

\begin{proposition}
\label{prop:bif}
Let $L$ be a link in $S^3$, and let $D$ be a chessboard-coloured diagram of $L$ with $l$ diagram components and $k$ nonalternating crossings.  Suppose that $L$ bounds a smoothly properly embedded surface $\Delta$ in the 4-ball with no closed components and with $\chi(\Delta)=k+l$.
Then $D$ is bifactorizable.
\end{proposition}

\begin{proof}
Let $X_1=\Sigma_2(D^4,F_b)$ be the double cover of the 4-ball branched along the black surface of the diagram $D$, and let $X_2=-\Sigma_2(D^4,\Delta)$ be the double-branched cover of the surface $\Delta$, with reversed orientation.  Let $Y=\partial X_1=-\partial X_2$ denote the double cover of $S^3$ branched along $L$.
The assumption that $\chi(\Delta)=k+l$ shows that in fact equality is attained in both \eqref{eq:etakl} and \eqref{eq:chieta}.
By \Cref{lem:semidef}, $X_1$ is positive semi-definite, and by \Cref{prop:slicelike}, $X_2$ is a rational homology 
$\displaystyle\natural_{\eta(L)}(S^1\times D^3)$.  

By the long exact sequences of the pairs $(X_i,Y)$ we have that the image of the inclusion-induced map $H_2(Y;\zz)\rightarrow H_2(X_1;\zz)$ is the nullspace of the intersection lattice $\Lambda_b$ of $X_1$, and the inclusion-induced map $H_1(Y;\qq)\rightarrow H_1(X_2;\qq)$ is an isomorphism.  Then the Mayer-Vietoris homology sequence with integer coefficients for $X=X_1\cup X_2$ shows that $H_2(X;\zz)$ contains the quotient of $\Lambda_b$ by its nullspace as a finite-index sublattice.  It follows that $X$ is a smooth closed positive-definite manifold, which then has a standard intersection lattice by Donaldson's diagonalisation theorem \cite{thmA}, so that the inclusion map of $X_1$ into $X$ induces a finite-index lattice embedding
$$\Lambda_b/\Lambda^0_b\rightarrow\zz^{b_2(X)},$$
as required.

The same argument applied to the mirror of $D$ with the other chessboard colouring yields a finite-index embedding of
$\Lambda_w/\Lambda^0_w$ in a standard diagonal lattice.
\end{proof}

The results above guide our search for ribbon disks for alternating knots.  An alternating knot has nullity zero, and a disk has Euler characteristic one.  More generally suppose we have a link $L$ with a chessboard-coloured diagram $D$, with $m+1$ white regions, $n+1$ black regions, $k$ nonalternating crossings, and $l$ diagram components, and admitting equality in \eqref{eq:etakl}, and we seek a slice surface $\Delta$ as in \Cref{prop:bif} with equality in \eqref{eq:chieta}.
We first check if the given diagram is bifactorizable; if not we know that no such slice or ribbon surface exists.  If it is bifactorizable we will attempt to find a ribbon surface $\Delta$ with $\chi(\Delta)=k+l$.  Such a surface admits a movie presentation in the form of a sequence of $b$ band moves and isotopies applied to the initial diagram and ending up with a diagram of the $(b+k+l)$-component unlink.  A band move can change the nullity of a link by at most one, and the nullity of the unlink exceeds that of $L$ by $b$, and thus each band move increases the nullity by one.  Each intermediate link in the sequence bounds the surface given by attaching bands as 1-handles to the black surface for the diagram $D$.  The double branched cover of this surface embeds in the closed positive-definite manifold $X=\Sigma_2(S^4,F_b\cup-\Delta)$ with $b_2(X)=m-k-l+1$ as in the proof of \Cref{prop:bif}.
A naive approach, which yields success surprisingly often, is to ask that such an embedding exists for the double branched cover of the black surface of the diagram obtained from $D$ by applying the given band moves.  A similar discussion applies to the white surface.  Thus our algorithm  seeks bands with the following property: they increase the nullity by one, and they preserve bifactorizability and also the ranks of the target diagonal lattices as in \eqref{eq:latmorph}.  It turns out that there are only finitely many such bands in any given diagram.

\begin{proposition}
\label{prop:algband}
Let $D$ be a bifactorizable diagram, and let $\beta$ be a band in standard position with respect to $D$.  Let $D'$ be the diagram resulting from $D$ by applying the band move determined by $\beta$.  Let $m+1,n+1,k,l,l_b,l_w$ (respectively, $m'+1,n'+1,k',l',l_b',l'_w$) denote the number of white regions, black regions, nonalternating crossings, diagram components, black surface components, and white surface components of $D$  (respectively, of $D'$).
Suppose that $D'$ is bifactorizable and the nullity of $D'$ is one greater than the nullity of $D$, and that the ranks as in \eqref{eq:latmorph} are preserved:
\begin{equation}
\label{eq:banddim}
\begin{aligned}
m'-k'-l'_w&=m-k-l_w,\\
n'-k'-l_b'&=n-k-l_b.
\end{aligned}
\end{equation}
Then  $\beta$ has length one, two, or three, and has twist number zero.  If $\beta$ is of length one or three, then it disconnects two connected summands; moreover if it is of length three, then it is isotopic via generalised Tsukamoto moves to an untwisted band of length one.  If it is of length two, then it is untwisted locally as well as globally, or in other words appears as in the right hand side of \Cref{fig:algbands}.\end{proposition}
\begin{proof}  A bifactorizable diagram gives rise to equality in \eqref{eq:etakl}.  Thus since the band move increases nullity by one, it follows that
\begin{equation}
\label{eq:bandkl}
k'+l'=k+l+1.
\end{equation}

Disconnected diagrams may be converted to connected diagrams by R2 moves or by planar connected sum bands, each of which also connects two components of one of the black and white surfaces; it follows as in the proof of \Cref{lem:latdim} that
$$l=l_b+l_w-1.$$
A similar argument combined with Euler's formula for connected planar graphs shows that the crossing number of $D$ is given by
$$\cross(D)=m+n-l+1.$$
Summing the equations in \eqref{eq:banddim} then yields
$$\cross(D')-\cross(D)=2(k'-k),$$
or in other words that the band $\beta$ introduces an even number of crossings, half of which are nonalternating.

Suppose now that the band $\beta$ has length $\lambda>1$.  This means that the band crosses arcs of the diagram $D$, over or under, $\lambda-1$ times, giving rise to $2\lambda-2$ crossings, half of which are nonalternating.  We claim that these are the only new crossings in $D'$, or in other words that the band $\beta$ has no twisting in any region.  This follows  from the fact that both Goeritz lattices of a bifactorizable diagram are positive semi-definite.  First note that the band $\beta$ cannot have an alternating twist adjacent to a nonalternating twist, as these would give rise to a 2-sided region having square zero in the Goeritz lattice, but nonzero intersection with two  adjacent regions, contradicting positive semi-definiteness.  Now suppose there is any twisting in the band.  This has to give rise to an equal number of alternating and nonalternating crossings, so in particular there must be some nonalternating crossing within the band that is adjacent to a crossing of the band over or under an arc of the diagram $D$; this gives rise to a triangular region with square $-1$ in the Goeritz lattice, again contradicting positive semi-definiteness. 

Thus the band $\beta$ introduces exactly $2(\lambda-1)$ crossings, from which we conclude that $k'-k=\lambda-1$, and then from \eqref{eq:bandkl} we have
$$l-l'=\lambda-2.$$
In other words, if the length $\lambda$ is greater than one, then the band reduces the number of diagram components by $\lambda-2$.

Choose an orientation of the core of the band so that we may refer to regions on the left and on the right.  Label the regions on the left $A_1$, $A_2$, \dots , $A_\lambda$ in order along the band, and label the regions on the right of the band $B_1$, $B_2$, \dots, $B_\lambda$.  For $1<i<\lambda$, the regions $A_i$ and $B_i$ are separated by a rectangle $C_i$ in the band, as in Figure \ref{fig:ABC5}, which has two alternating crossings and two nonalternating crossings and hence square zero in the Goeritz lattice.  It follows from semi-definiteness that $C_i$ cannot have nonzero pairing with any other element of the lattice.  We see that $C_i$ shares a crossing with each of $A_{i-1}$, $B_{i-1}$, $A_{i+1}$, and $B_{i+1}$, from which it follows that these four regions are not in fact distinct in the diagram; they are connected in such a way that the Goeritz pairings with $C_i$ cancel out.

\begin{figure}[htbp]
\centering
\includegraphics[scale=0.6]{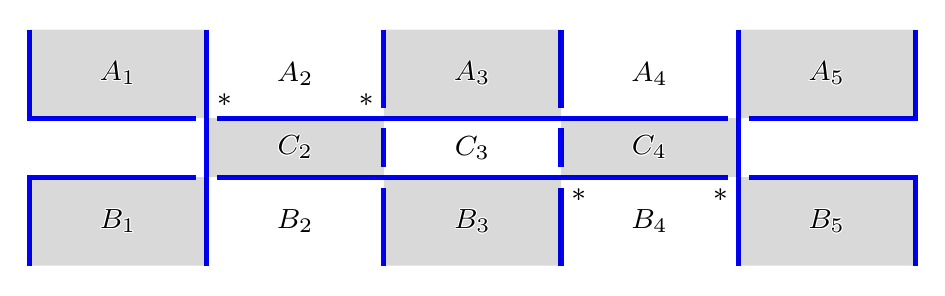}
\begin{narrow}{0.3in}{0.3in}
\caption
{{\bf Regions around a band.} }
\label{fig:ABC5}
\end{narrow}
\end{figure}

We now establish that this rules out bands of length greater than 3.  If $\lambda>3$, then we cannot have $A_1$ and $A_3$ being the same region, as this would isolate $A_2$ giving a nonzero pairing between $A_2$ and $C_3$, contradicting semi-definiteness.  In the same way we cannot have $B_1=B_3$, $A_1=B_3$, or $A_3=B_1$.  The only remaining way to have the regions connect to give zero pairing with $C_2$ is to have $A_1=B_1$ and $A_3=B_3$.  The same argument applies along the band to show that in fact we must have $A_i=B_i$ for $i=1,\dots,\lambda$.  This implies that in fact the band reduces the number of diagram components by $\lambda$, which is a contradiction.

We next consider an untwisted band of length $\lambda=3$.  Label the regions as in Figure \ref{fig:ABC3}.  Again we see there is a rectangular region $C_2$ in the band which has square zero and thus cannot have nonzero pairing with any other region.  Suppose now that $C_2$ has two nonalternating crossings on the same side of the band (left or right), as in the first diagram in Figure \ref{fig:ABC3}.  It follows that $A_1=B_1$ and $A_3=B_3$ (and possibly all four are the same region).  This implies that the band reduces the number of diagram components by at least two, which is a contradiction.

\begin{figure}[htbp]
\centering
\includegraphics[scale=0.6]{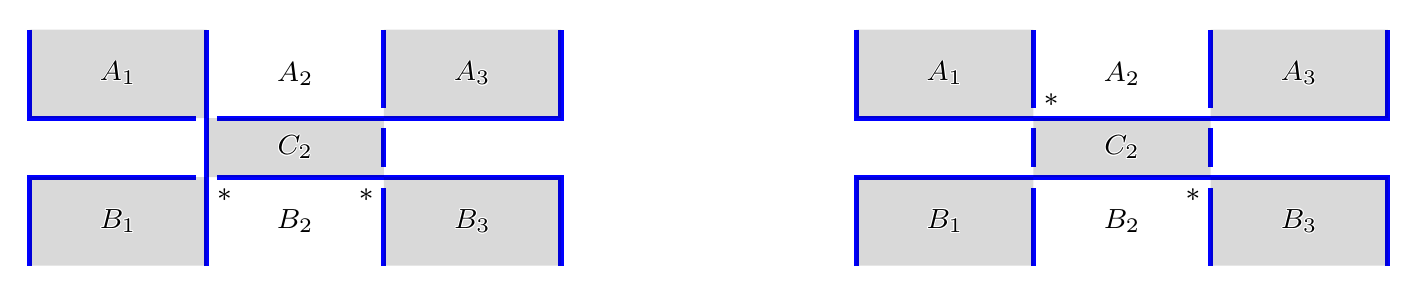}
\begin{narrow}{0.3in}{0.3in}
\caption
{{\bf Regions around bands of length three.} }
\label{fig:ABC3}
\end{narrow}
\end{figure}

Suppose then that we have a band of length 3 as in the second diagram in Figure \ref{fig:ABC3}, with one nonalternating crossing on each side of the band.  We again see that $A_1=B_1$ implies $A_3=B_3$ and the band reduces the number of diagram components by at least two.  We conclude that $A_1=A_3\ne B_1=B_3$, giving a band of the form shown in Figure \ref{fig:length3}, which disconnects two connected summands.  Moreover, this band can be isotoped to a length one band, as shown in Figure \ref{fig:length3}, by a sequence of generalised Tsukamoto moves (two flypes and two R2 moves).

\begin{figure}[htbp]
\centering
\includegraphics[scale=0.6]{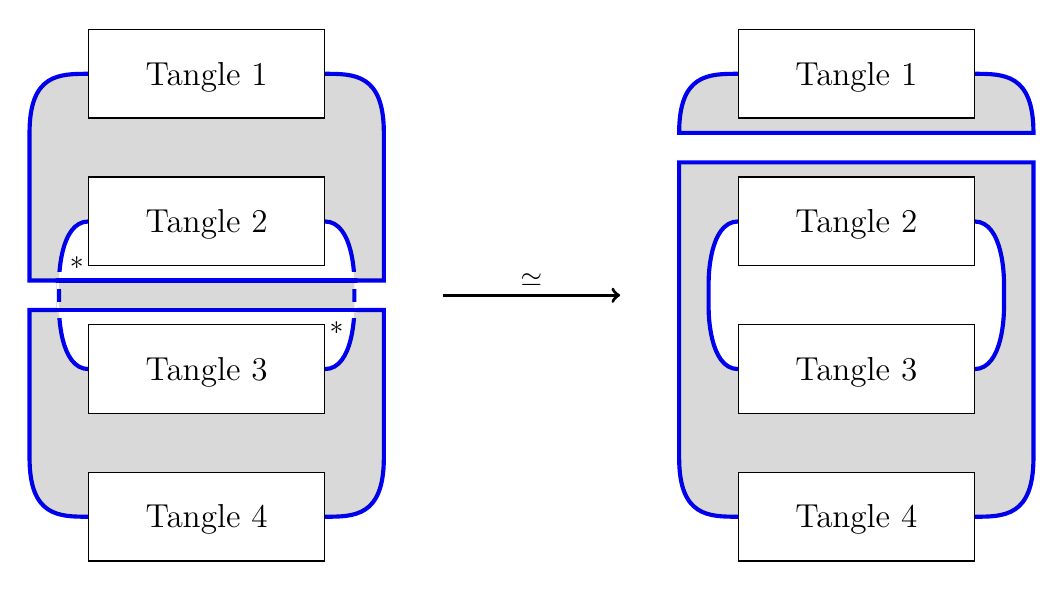}
\begin{narrow}{0.3in}{0.3in}
\caption
{{\bf An uninteresting band of length three.} }
\label{fig:length3}
\end{narrow}
\end{figure}

Finally we consider the case of length one.  Label the regions adjacent to the band as in Figure \ref{fig:length1}.  
Note that as above the band must introduce the same number of alternating and nonalternating crossings, and one cannot have two adjacent nonalternating crossings as these would lead to a 2-sided region with negative square in the Goeritz lattice, contradicting positive semi-definiteness.  Any two adjacent crossings in the band then yield a 2-sided region of square zero, which cannot have nonzero pairing with any other region.  It follows that the band either has no twisting, with $k'=k$, or has one alternating and one nonalternating crossing, with $k'=k+1$, and the latter case can only occur if $C_1=C_2$.

We first consider the case $C_1=C_2$.  As above, the band is either untwisted or has one alternating and one nonalternating crossing.  These two possibilities are related by an R2 move, and both have the effect of disconnecting a connected sum.

\begin{figure}[htbp]
\centering
\includegraphics[scale=0.6]{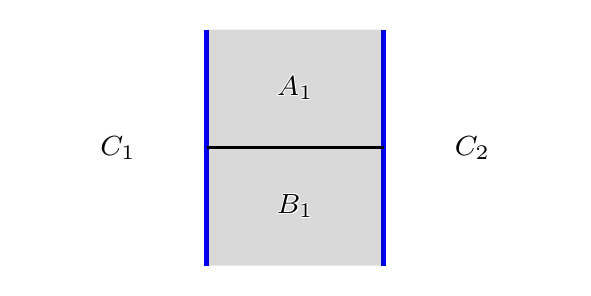}
\begin{narrow}{0.3in}{0.3in}
\caption
{{\bf Regions around a band of length one.} The horizontal arc is the core of the band. }
\label{fig:length1}
\end{narrow}
\end{figure}

The second case to consider is $A_1=B_1$.  This implies $l'=l-1$ and then $k'=k+2$ by \eqref{eq:bandkl}, which we have seen is impossible.  The final case is $A_1\ne B_1$ and $C_1\ne C_2$.  This implies $l'=l$ and $k'=k+1$, but we have seen that this can only happen if $C_1=C_2$.
\end{proof}

The preceding proposition motivates the definition of algorithmic bands given in \Cref{def:algband}.  A couple of remarks are in order.  First of all, we do not include untwisted bands of length one or three which disconnect a connected sum, even though these appear in the statement of \Cref{prop:algband}.  This is because, based in part on consideration of examples, we do not expect these to be useful in finding ribbon disks.  Algorithmic ribbon disks for prime alternating knots may indeed have intermediate links which are nonprime.  However, an algorithmic ribbon disk incorporating such a band $\beta$ will include subsequent algorithmic bands which convert each of the connected summands to unlinks, and we expect that these bands would suffice without the band move determined by $\beta$. Secondly, we include bands of length one and twist number one, with one nonalternating crossing in the band, and for which the resulting diagram $D'$ is bifactorizable.  This is simply a computational short cut; such a band move is equivalent to a length two untwisted algorithmic band followed by an R1 move.  Finally it is interesting to compare our resulting candidate algorithmic bands as in Figure \ref{fig:algbands} to those appearing in a slightly different context in \cite[Conjecture 11.2 and Figure 31]{BL}.

We now recall a theorem of Tsukamoto and state a conjectured generalisation of it.  These are intended to offer some insight into the success of the generalised Tsukamoto moves in the algorithm.  These moves are typically necessary after each band move in order to find the next band move, or to simplify to a crossingless diagram when a band move results in the unlink.

Tsukamoto used variants of the moves in Figures \ref{fig:flype}, \ref{fig:AAreidemeister}, and \ref{fig:untongue} to simplify almost-alternating diagrams of the unknot and 2-component unlink \cite{tsuk2,tsuk1}.  In particular, the main theorem of \cite{tsuk2} implies the following.  Here a $k$-almost alternating diagram is a diagram with $k$ nonalternating crossings.

\begin{theorem}
\label{thm:tsuk}
Every 1-almost alternating diagram of the 2-component unlink may be converted to the standard crossingless diagram by a finite sequence of flypes, untongue moves, Reidemeister 1 moves and a single Reidemeister 2 move.
\end{theorem}

We conjecture a generalisation of Tsukamoto's result, which is relevant to the algorithm we describe later in this paper.

\begin{conjecture}
\label{conj:unlink}
Let $k$ be a natural number.  Then there exists a finite set \,$\calu_k$ of generalised Tsukamoto moves such that each $k$-almost alternating diagram of the $(k+1)$-component unlink may be converted to the standard crossingless diagram by a finite sequence of moves from $\calu_k$.
\end{conjecture}

%%%%%%%%%%%%%%%%%%%%%%%%%%%%%%%%%%%%%%%%%%%%%%%%%%%%%%%%%%%%%%%%%%%%%%%%%%

\section{Examples}
\label{sec:examples}

In this section we give some examples of algorithmically ribbon knots and links, and state a conjecture.  Further examples may be found on the accompanying web page at \url{www.klo-software.net/ribbondisks}.

A recent paper of Brejevs has made use of our algorithm in the study of alternating 3-braid closures \cite{brejevs}; he exhibits ribbon surfaces with Euler characteristic one for several families of these.
Although not explicitly verified there, it seems reasonable to conclude that all alternating 3-braid closures which are currently known to be $\chi$-slice are in fact algorithmically ribbon.

\begin{proposition}
\label{prop:lisca1}
Two-bridge links which bound smoothly embedded Euler characteristic one surfaces in the 4-ball without closed components are algorithmically ribbon.
\end{proposition}

\begin{proof}
For the definition and notation for two-bridge links, see \cite{lisca1} or \cite{balls}. Two-bridge links which bound smoothly embedded Euler characteristic one surfaces in the 4-ball without closed components were classified and shown to bound ribbon surfaces by Lisca \cite{lisca1}.   Up to mirroring, these may be divided into six families as follows:
\begin{enumerate}[(a)]
    \item $S(m^2,mq\pm1)$, with $m>q>0$ and $(m,q)=1$;
    \item $S(m^2,mq\pm1)$, with $m>q>0$ and $(m,q)=2$;
    \item $S(m^2,d(m-1))$, with $m^2/(m-1)>d$ and $d$ divides $2m+1$;
    \item $S(m^2,d(m-1))$, with $m^2/(m-1)>d$ and $d$ odd divides $m-1$;
    \item $S(m^2,d(m+1))$, with $m^2/(m+1)>d$ and $d$ odd divides $m+1$;
    \item $S(m^2,d(m+1))$, with $m^2/(m+1)>d$ and $d$  divides $2m-1$.
\end{enumerate}
A band move leading to a ribbon surface of Euler characteristic one is shown for each of these in \Cref{fig:Liscadisks} (cf. \cite{bbl,lisca1}).  It remains to verify that these surfaces are in fact algorithmically ribbon. This is shown in detail for family (d) in \Cref{fig:alg2bridge}; the other cases are similar and left to the reader.
\end{proof}

\begin{figure}[htbp]
\centering
\includegraphics[scale=0.3]{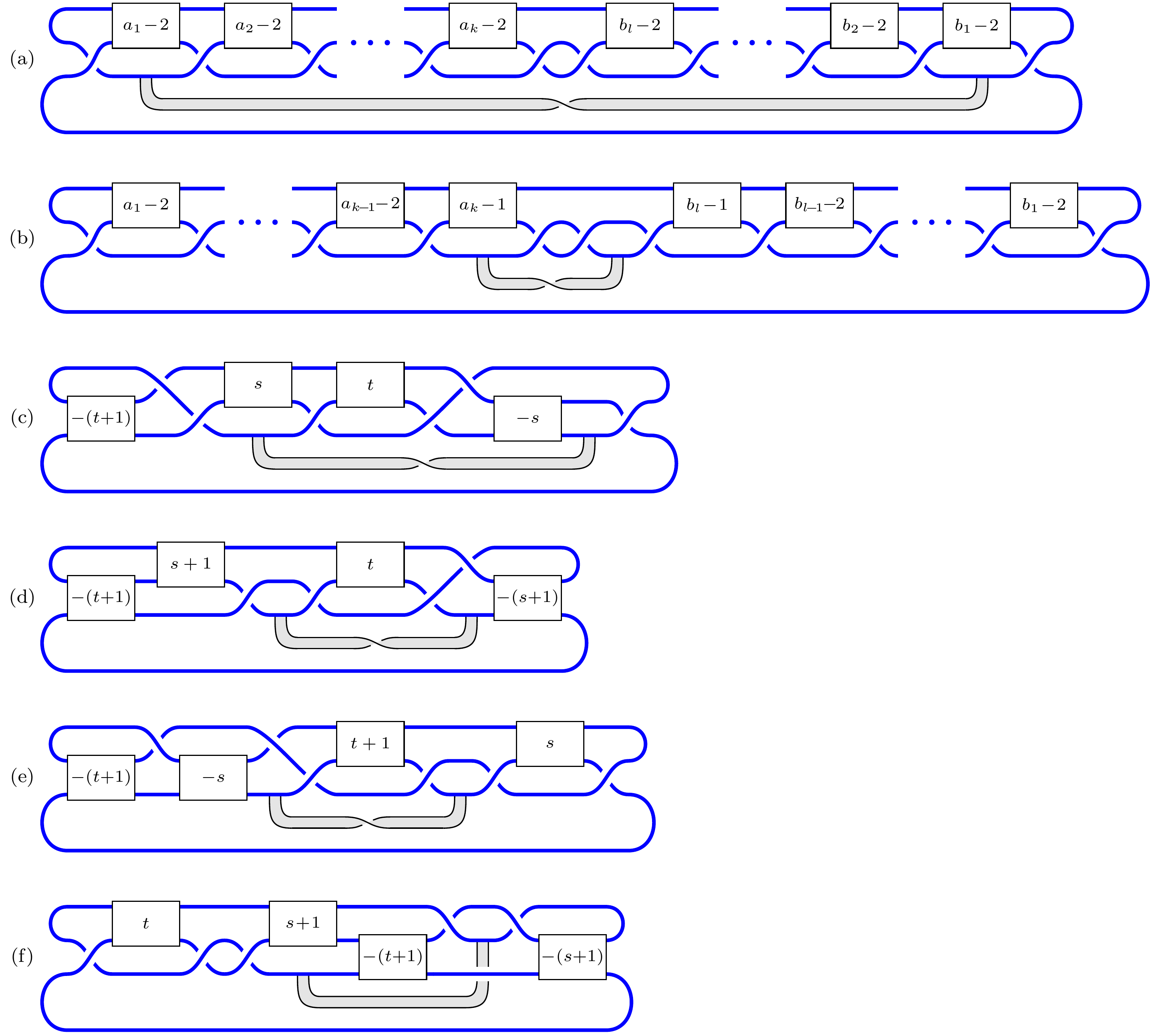}
\begin{narrow}{0.3in}{0.3in}
\caption
{{\bf Two-bridge links and ribbon surfaces.} Here $s$ and $t$ can be any nonnegative integers, while $m/q=[a_1,\dots,a_k]:=a_1-\dfrac{1}{\displaystyle{a_2}-\raisebox{-3mm}{$\ddots$
\raisebox{-2mm}{${-\dfrac{1}{\displaystyle{a_k}}}$}}}$ and $m/(m-q)=[b_1,\dots,b_l]$.}
\label{fig:Liscadisks}
\end{narrow}
\end{figure}

\begin{figure}[htbp]
\centering
\includegraphics[scale=0.3]{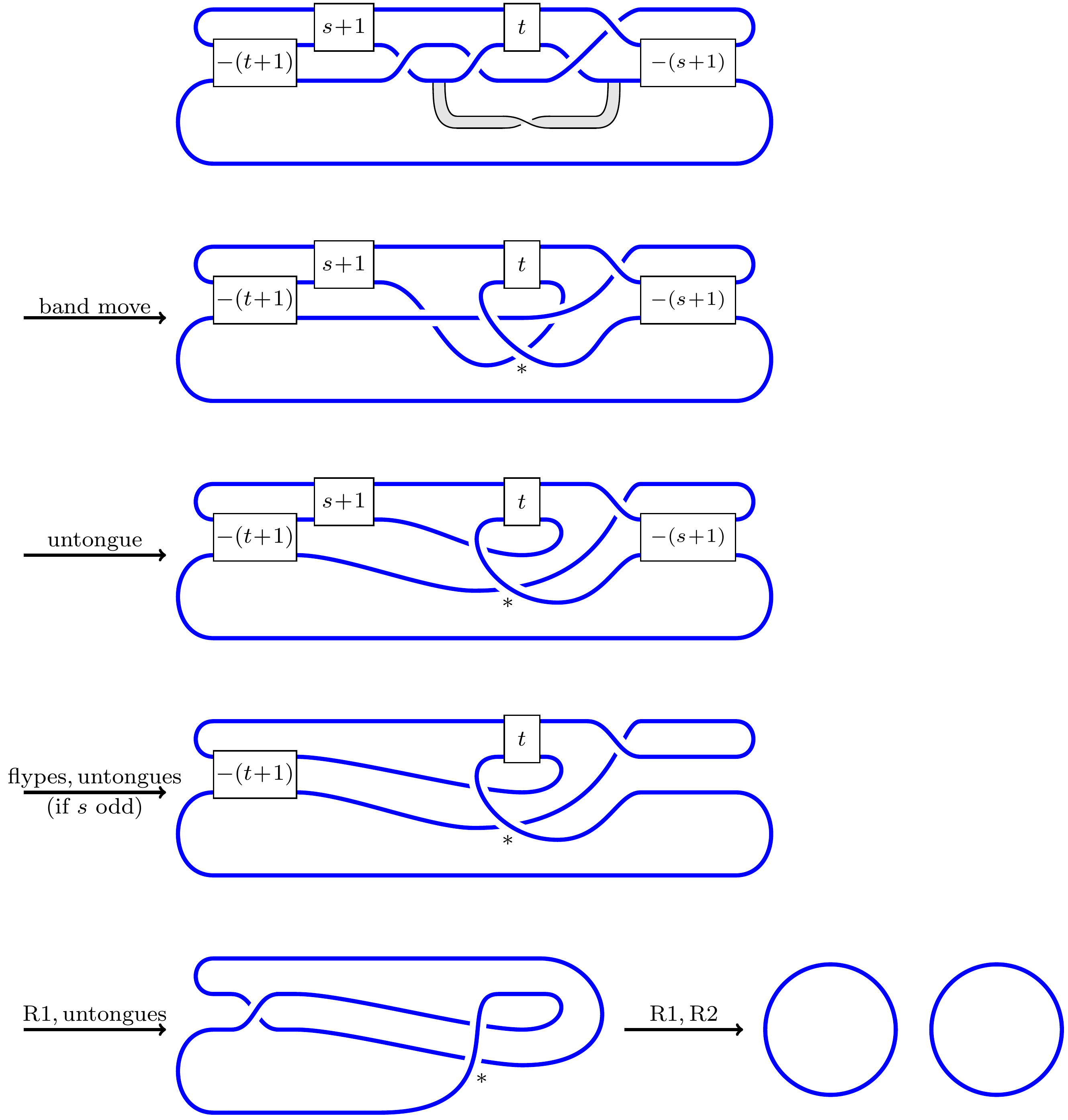}
\begin{narrow}{0.3in}{0.3in}
\caption
{{\bf algorithmic ribbon surfaces for 2-bridge knots and links, family (d).} For $s$ even, there is a minor modification to make in the fourth diagram.  Note that one can easily obtain the bifactorizations for each diagram by working back from the crossingless diagram at the bottom.} 
\label{fig:alg2bridge}
\end{narrow}
\end{figure}

\begin{proposition}
\label{prop:Lsum-L}
Let $L$ be an oriented nonsplit alternating link with a marked  component, and let $-L$ denote the mirror of $L$, with reversed orientation.  The connected sum $-L\#L$, using the marked components to take connected sum, is algorithmically ribbon.
\end{proposition}

\begin{proof}
Draw an alternating diagram of $-L\#L$ as shown in \Cref{fig:Lsum-L}.  We will describe a bifactorization of this diagram, or equivalently lattice embeddings
\begin{equation}
\label{eq:Lsum-L}
\begin{aligned}
\Lambda_b&\rightarrow\zz^{n},\quad\mbox{and}\\
\Lambda_w&\rightarrow\zz^{n}
\end{aligned}
\end{equation}
from the Goeritz lattices to diagonal unimodular lattices whose rank $n$ is equal to the crossing number of $L$.
Label the crossings in the tangle $\Gamma$ from $1$ to $n-1$, and label each crossing in the tangle $\overline{\Gamma}$ with the same label as its reflected image in $\Gamma$.  Label both crossings shown in \Cref{fig:Lsum-L} with $n$.  The embeddings \eqref{eq:Lsum-L} are given by summing local contributions in each region, where the crossing labelled $i$ gives the local contributions shown in \Cref{fig:localemb}. One can alternatively reverse the steps in the sequence described in the proof and recover the bifactorization step by step from that for the crossingless unlink, as discussed after \Cref{def:algsurf}.

We now apply a sequence of band moves as shown in \Cref{fig:sumband}, one for each of the $n-1$ crossings in the tangle $\Gamma$.
There are two choices for where to apply the first of these, depending on which crossing in \Cref{fig:Lsum-L} will take the role of the central crossing in the left side of \Cref{fig:sumband}.  Each such band move removes one crossing from the original tangle $\Gamma$, and symmetrically from $\overline{\Gamma}$.  Outside the modified tangles, two new alternating crossings are introduced on either side of a previously alternating crossing, which becomes nonalternating, as in Figures \ref{fig:sumband} and \ref{fig:Lsum-Lbanded}.

We need to see that we may apply $n-1$ such bands and thus replace $\Gamma$ with a crossingless tangle.  Note from \Cref{fig:sumband} that a subsequent band coming into $\Gamma'$ from one of the new nonalternating crossings outside $\Gamma'$ as in \Cref{fig:Lsum-Lbanded} may turn right or left when it comes to the site of any crossing removed by a previous band.  It will then be constrained to follow a path through the original tangle $\Gamma$ which keeps black regions to the right and white regions to the left.  We therefore need to see that every crossing in $\Gamma$ is connected to either the top or bottom incoming strand by a zig-zag path that turns right or left at each crossing it encounters, keeping black regions on the right.  First note that since $L$ was assumed to be nonsplit, the tangle $\Gamma$ has at most two diagram components, so that every crossing is connected to either the top or bottom incoming strand by a path in the diagram.  Then given any edge in such a path with a white region on the right, replace it in the path by the rest of the edges of the same white region, which will then have black on the right and white on the left.

The result of these $n-1$ band moves is shown in \Cref{fig:Lsum-Lbanded}.  
We claim that this is in fact a diagram of the $n$-component unlink, and converts to the crossingless diagram by a sequence of Reidemeister 1 and 2 moves as in \Cref{fig:AAreidemeister}.

There are a total of $2n+2$ strands entering the crossingless tangle $\Gamma'$: one each on the right and left, and the rest on the top or bottom.  Suppose first that any pair of strands coming into the top are connected inside $\Gamma'$.  Choose an innermost such pair, which are necessarily adjacent.  By symmetry the same pair of strands is connected inside $\overline{\Gamma'}$ and we may apply a Reidemeister 2 move as in \Cref{fig:AAreidemeister} to split off a crossingless unknot.  The same applies to strands on the bottom; thus after splitting off some crossingless unknots, we may assume that all strands in $\Gamma'$ enter and leave by a different side of the tangle.  The strand entering $\Gamma'$ from the left then exits via the leftmost strand on the top or bottom, and the strand entering $\Gamma'$ from the right exits via the rightmost strand on the top or bottom, and these choices completely determine $\Gamma'$.  In each case we find the resulting diagram simplifies to a crossingless unlink via two Reidemeister 1 moves and a number of Reidemeister 2 moves as in \Cref{fig:AAreidemeister}.  

The sequence described above, consisting of $n-1$ algorithmic band moves, $n-1$ untongue moves, $n-1$ Reidemeister 2 moves, and two Reidemeister 1 moves, represents an algorithmic ribbon surface for $-L\#L$.
\end{proof}

\begin{figure}[htbp]
\centering
\includegraphics[scale=0.6]{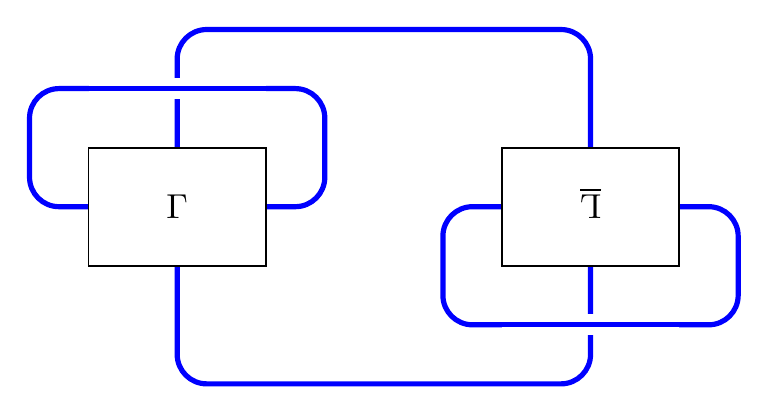}
\begin{narrow}{0.3in}{0.3in}
\caption
{{\bf An alternating diagram of $-L\#L$.}  The tangle marked $\overline{\Gamma}$ is the image of the tangle marked $\Gamma$ under reflection in a plane perpendicular to the plane of the diagram.} 
\label{fig:Lsum-L}
\end{narrow}
\end{figure}

\begin{figure}[htbp]
\centering
\includegraphics[scale=0.8]{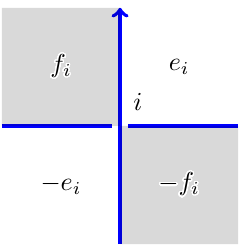}
\begin{narrow}{0.3in}{0.3in}
\caption
{{\bf Local contributions to lattice embeddings.}  The orientation on the overcrossing comes from that on $-L\#L$. } 
\label{fig:localemb}
\end{narrow}
\end{figure}

\begin{figure}[htbp]
\centering
\includegraphics[scale=0.7]{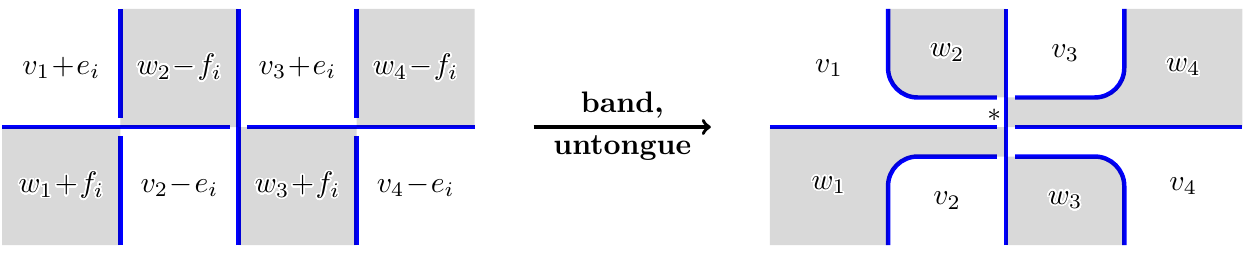}
\begin{narrow}{0.3in}{0.3in}
\caption
{{\bf A band move.} This is equivalent to composing a length two algorithmic band placed above or below the horizontal strand with an untongue move.  } 
\label{fig:sumband}
\end{narrow}
\end{figure}

\begin{figure}[htbp]
\centering
\includegraphics[scale=0.8]{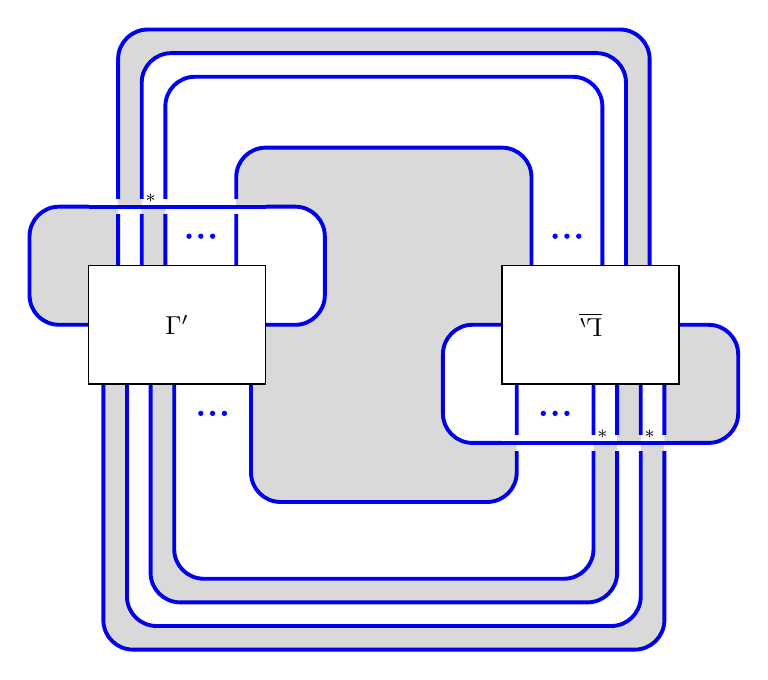}
\begin{narrow}{0.3in}{0.3in}
\caption
{{\bf The result of band moves on $-L\#L$.}  The tangle marked $\overline{\Gamma'}$ is the image of the tangle marked $\Gamma'$ under reflection in a plane.  After $n-1$ band moves, the tangles $\Gamma'$ and $\overline{\Gamma'}$ are crossingless.} 
\label{fig:Lsum-Lbanded}
\end{narrow}
\end{figure}

We conclude this section with an optimistic conjecture.  This is based on experimental evidence\footnote{It has been directly verified via KLO that no band of length at most $3$ and absolute twist at most $2$ on any non-algorithmically-ribbon alternating knot up to $17$ crossings results in an unlink.} and is perhaps encouraged by a comparable theorem of McCoy about unknotting numbers of alternating knots \cite{mccoy}.  

\begin{conjecture}
\label{conj:fusion1}
Let $L$ be an alternating link which bounds a ribbon surface with two minima and a single saddle point.  Then $L$ is algorithmically ribbon.
\end{conjecture}

%%%%%%%%%%%%%%%%%%%%%%%%%%%%%%%%%%%%%%%%%%%%%%%%%%%%%%%%%%%%%%%%%%%%%%%%%%

\section{Computer Implementation}
\label{sec:imp}

Code pertaining specifically to the algorithm described above (henceforth referred to as the Algorithm) was written in C++ as a custom computational module for the Knot-Like Objects (KLO) software project by F.~Swenton \cite{klo}, which had pre-existing functionality for dealing with knots and bands on knots.  Specific methods were coded for computing Goeritz matrices for a given diagram, searching for factorisations of those matrices, and searching for algorithmic bands to add to a given knot such that the band result has a bifactorisation extending that of the original knot.  Additionally, the generalised Tsukamoto moves, here referred to as Almost-Alternating (AA) simplification moves, were implemented in order to algorithmically simplify the results of the bands, and code was added to find all flyped forms of a diagram.
Broadly, the Algorithm functions as a breadth-first search, as follows:

\vspace{6pt}
{\small
\noindent
Input: A knot diagram (initially, of an alternating knot of a single component)

\begin{enumerate}
\item%
Find both Goeritz matrices for the diagram.

\item%
Declare the knot non-slice if either matrix does not factorise. %ribbon if either matrix does not factorise.

\item%
Otherwise, for each flype-equivalent diagram:

\begin{enumerate}

\item%
Find all factorisations of each of the two Goeritz matrices, up to $\mathop{\mathrm{Aut}}(\mathbb{Z}^n)$-symmetries (note that a factorization of one flyped form allows direct computation of those of the rest, greatly speeding up this step).

\item%
For each candidate algorithmic band as in Figure \ref{fig:algbands}:

\begin{enumerate}
\item%
Test whether the band is \emph{algorithmic}, i.e., whether the assignment of Goeritz $\mathbb{Z}$-vectors for the faces of the diagram unaffected by the band can be extended into a Goeritz bifactorisation for the entire diagram resulting from the band operation.

\begin{enumerate}
\item%
If so, AA-simplify the resulting diagram.  If it simplifies to an unlink diagram with no crossings, declare a ribbon knot to be found; otherwise, add the resulting diagram as a child of the starting diagram.

\item%
If not, declare a dead-end at the resulting diagram.
\end{enumerate}  
\end{enumerate}

\end{enumerate}
\end{enumerate}

\noindent 
Iterate the Algorithm on all children until a ribbon knot is declared to have been found or only dead ends remain.

\vspace{6pt}

\noindent Output:~~\begin{minipage}[t]{5in}
The knot is declared nonslice in Step (2) if the alternating diagram is not bifactorisable.  If the initial diagram is bifactorisable, then if any descendant is identified as ribbon, the precise sequence of flypes, algorithmic bands, and simplifications needed to produce an unlink from the original diagram are exhibited, and the original knot is declared \emph{algorithmically ribbon}.
Otherwise, the knot is declared \emph{algorithmically non-ribbon} (we note that a comparatively small number of examples show this algorithm to be incomplete in its current form---i.e., some alternating ribbon knots will be declared as \emph{algorithmically} non-ribbon---which is why this distinction is made).
\end{minipage}
}

\vspace{6pt}

A few notes on this implementation are in order.  First, checking all variants of the 3-flype move in Figure~\ref{fig:tripleflype} in the original algorithm would result in a branching of the search tree, so to simplify matters, a complete floating loop over the relevant 3-tangle is introduced instead, as in Figure~\ref{fig:loopthreeflype}, which also preserves bifactorisability and only differs in that it has one additional crossing.  As no subsequent bands will connect to this loop, it is clear that any knot found to be algorithmically ribbon with the floating loop will still, in fact, be ribbon (simply remove it).  It is not \emph{a priori} clear that, in general, this version of the 3-flype yields the same results in the Algorithm as do the other versions of the 3-flype (or, indeed, that they produce the same results as one another), but it has been directly verified that each of the knots encountering a 3-flype but \emph{not} identified as ribbon in the computation below is also not algorithmically ribbon via any of the non-loop versions of the 3-flype.

\begin{figure}[htbp]
\centering
\includegraphics[scale=0.5]{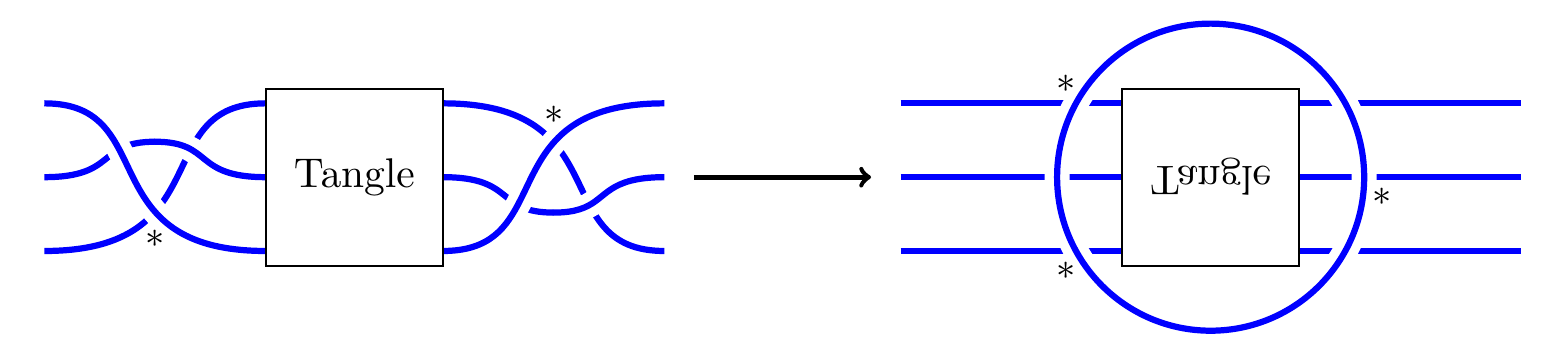}
\begin{narrow}{0.3in}{0.3in}
\caption
{\bf The ``floating loop'' 3-flype move.} 
\label{fig:loopthreeflype}
\end{narrow}
\end{figure}
%%%%%%%%%%%%%%%%%%%%%%%%%%%%%%%%%%%%%%%%%%%%%%%%%%%%%%%%%%%%%%%%%%%%%%%%%%

\section{Results}
\label{sec:results}

The Prime Alternating Knot Generator (PAKG) \cite{PAKG} from Flint, Rankin, and de Vries was used to generate complete enumerations of prime alternating knots up to 21 crossings; the Algorithm was applied to those having square determinant.  Table~\ref{ResultsTable} summarizes the results of the computation, listing by crossing range: the number of prime alternating knots (PAK) up to reflection; the number of these with square determinant ($\square$ det); the number with Goeritz bifactorisations; the number found to be algorithmically ribbon (AR); and the numbers of bands that the Algorithm takes to decompose these knots into unlinks (i.e., the number of index-one critical points in the ribbon disk that the Algorithm exhibits for these knots).

\begin{table}[!h]
\centering
\begin{tabular}{r|r|r|r|r|l}
\textbf{$\times$'s} & \textbf{PAK} & \textbf{$\square$~det} & \textbf{Bifac} & \textbf{AR} & \textbf{Bands} \\\hline
\vrule width 0pt height 13pt to 11 & 563 & 36 & 28 & 28 & $1^{28}$ \\ 
12 & 1,228 & 62 & 51 & 48 & $1^{47}\cdot2^1$ \\
13 & 4,878 & 175 & 138 & 118 &  $1^{118}$ \\
14 & 19,536 & 567 & 409 & 305 & $1^{301}\cdot 2^4$ \\
15 & 85,263 & 1,921 & 1,245 & 850 & $1^{805}\cdot 2^{45}$ \\
16 & 379,799 & 6,888 & 3,724 & 2,330 & $1^{2,033}\cdot 2^{297}$ \\
17 & 1,769,979 & 24,828 & 11,259 &6,513 & $1^{5,384}\cdot 2^{1,129}$ \\
18 & 8,400,285 & 91,486 & 34,197 & 19,071 & $1^{13,731}\cdot 2^{5,333} \cdot 3^7$\\
19 & 40,619,385 & 349,453 & 99,429 & 52,325 & $1^{35,270}\cdot 2^{17,039} \cdot 3^{16}$\\
20 & 199,631,989 & 1,322,355 & 287,923 & 150,408 & $1^{87,623}\cdot 2^{62,301} \cdot 3^{484}$\\
21 & 990,623,857 & 5,258,538 & 864,649 & 430,907 & $1^{225,985}\cdot 2^{203,320} \cdot 3^{1,602}$\\\hline
\vrule width 0pt height 13pt {Total} & {1,241,536,199} & {7,056,273} & {1,303,024} & {662,903} & ${1^{371,325}\cdot2^{289,469}\cdot 3^{2,109}}$
\end{tabular}
\vspace{6pt}

\caption{{\bf Computational results of the Algorithm.} In the rightmost column, $a^b$ indicates that $b$ of these knots have algorithmic ribbon disks with not less than $a$ band moves.}
\label{ResultsTable}
\end{table}

The Algorithm thus directly yields 662,903 prime alternating ribbon knots up to 21 crossings, exceeding by far the previous state of the art (see for example \cite{knotinfo}). We note that the first occurrence of a knot requiring two bands is at 12 crossings, and the first requiring three bands is at 18 crossings.  Also, we see a diminishing proportion of Goeritz-bifactorisable knots being algorithmically ribbon as the crossing number increases: this starts with 28/28 being algorithmically ribbon up to 11x and 48/51 at 12x---the three left out being 12a631 (ribbon, discussed in the next section), 12a360, and 12a1237 (both nonslice \cite{knotinfo})---and ends with fewer than half at 21x.

\section{Intermediate forms and escapees}
\label{sec:escapees}

The Algorithm explicitly exhibits any algorithmically ribbon knot as being, in fact, a ribbon knot; however, when the Algorithm's search fails, we have no absolute assurance that the knot in question is \emph{not} ribbon.  A small number of examples are known of alternating ribbon knots that the Algorithm fails to identify as such---indeed, Seeliger \cite{Seeliger} exhibits the knots 12a631, 14a6639, and 14a7977 as ribbon.  These knots provide our first three examples of what we'll call  \emph{escapees} [from the Algorithm].  We shall see that these knots do not entirely escape the Algorithm---they just make a detour at the start of their journey (though a complete understanding of this phenomenon has proved to be elusive).

A preliminary search for escapees began with the search data up to 18 crossings by extracting every non-split diagram of two components obtained along a path (of flypes, bands, and AA-simplifications) from an alternating knot that was identified as ribbon, resulting in a list of 17,739 diagrams.  We then enumerated all fusion bands, of length up to $4$ and absolute twist up to $2$, on these knots and tested the results of these bands for hyperbolic isometry with all hyperbolic prime alternating knots up to 18 crossings that were potential escapees (i.e.,  Goeritz-bifactorisable but not algorithmically ribbon).

In the course of this preliminary search, it was noted that all escapees so identified were obtainable via certain \emph{escape bands} of length 3, proceeding as in \Cref{fig:escapebands}, either entirely over or under with no twist (left) or over and under with an alternating half-twist (right).  It is worth noting that these band types introduce into an alternating knot exactly \emph{two} non-alternating crossings, and thus they are not quite algorithmic, though they are as close as possible to being so.  While the diagrams immediately resulting from these bands cannot bifactor,\footnote{One of the Goeritz lattices of the diagram resulting from an escape band is positive semi-definite, and it follows from Donaldson's theorem that if this band is part of the sequence for a ribbon disk then the quotient of this lattice by its nullspace admits a finite-index embedding in a diagonal lattice.  Thus the methods of this paper can be applied to obstruct some escape bands; however after applying unobstructed escape bands, one does not obtain a minimally nonalternating diagram and the methods described up to now do not apply.} sequences of Reidemeister moves were found for each diagram that transform it into a bifactoring 2-component link diagram that the Algorithm identifies as ribbon.

\begin{figure}[htbp]
\centering
\includegraphics[scale=0.6]{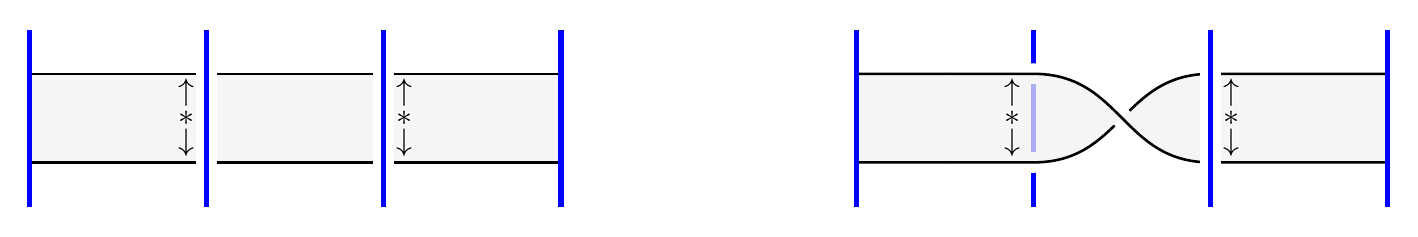}
\begin{narrow}{0.3in}{0.3in}
\caption
{{\bf Escape bands.} }
\label{fig:escapebands}
\end{narrow}
\end{figure}

It should be stressed that the Algorithm is entirely combinatorial and that it bypasses any issue of knot (or unlink) identification, because the Algorithm terminates either when an unlink diagram---i.e., a diagram with no crossings---is obtained or when all of a knot's descendants in the Algorithm are dead ends having no algorithmic bands.  Our search for escapees lacks this advantage, so we used geometric methods as an aid.\footnote{All computations involving hyperbolic structure and hyperbolic knot isometry were performed via the the SnapPea kernel \cite{SnapPea}, which is included in the KLO software.}  Of the 662,903 algorithmically ribbon knots identified, 371,325 require just one band, i.e., after one algorithmic fission, they AA-simplify to an unlink.  The remaining 291,578 each produce a bifactorising link of two components requiring one or (far less frequently) two more bands in the Algorithm before an unlink is produced.  Filtering the 2-component results of the \emph{first} bands on these algorithmically ribbon knots yields 22,477 distinct 2-component \emph{hyperbolic intermediates}, which provide targets for results of escape bands.

To identify candidates for escapees, all 640,149 Goeritz-bifactorising prime alternating knots to 21 crossings \emph{not} identified as ribbon by the Algorithm were checked against two known obstructions to sliceness.  First, the Fox-Milnor condition on the Alexander polynomial \cite{FM} was checked via KLO, with explicit factorisations found that demonstrate 237,909 of these knots not to be slice.  The remaining knots were tested for the Herald-Kirk-Livingston obstruction \cite{HKL}, as implemented in SnapPy \cite{snappy} and SageMath \cite{sage}, for a range of pairs $(p,q)$ of primes, eliminating 393,793 more knots; additionally, one knot at 14 crossings is the closure of the 3-braid $(\sigma_1\sigma_2^{-1})^7$, which is not ribbon \cite{ammmps,sartori}.\footnote{\,The specific obstructions for all non-escapees are recorded at the site containing the search results.}  This leaves 5,006 as neither identified as ribbon by the Algorithm nor obstructed from sliceness; we enumerated all escape bands on each of these knots.  The escape band results that were hyperbolic were tested for hyperbolic isometry with those (if any) of the hyperbolic intermediates of the appropriate hyperbolic volume, and the bands producing such matches were recorded.

\begin{table}[!h]
\centering
\begin{tabular}{r|r|r|r|r|c|r}
\textbf{$\times$'s} & \textbf{Non-AR} & \textbf{Obstr.} & \textbf{Candid.} & \textbf{Escapees} & \textbf{Bands} & \textbf{Unres.} \\\hline
\vrule width 0pt height 13pt 
12 & 3 & 2 & 1 & 1 & $2^1$\\
13 & 20 & 20 & &  & \\
14 & 104 & 100 & 3 & 3 & $2^3$  \\
15 & 395 & 394 & 1 & 1 & $2^1$ \\
16 & 1,394 & 1,365 & 29 & 23 & $2^{23}$ & 6 \\
17 & 4,746 & 4,699 & 47 & 46 & $2^{46}$  & 1 \\
18 & 15,126 & 14,999 & 127 & 91 & $2^{91}$ & 36 \\
19 & 47,104 & 46,795 & 309 & 262 & $2^{257}\cdot 3^{5}$ & 47\\
20 & 137,515 & 136,537 & 978 & 482 & $2^{480}\cdot 3^{2}$ & 496 \\
21 & 433,742 &  430,231 & 3,511 & 821 & $2^{814}\cdot 3^{7}$ & 2,690\\\hline
\vrule width 0pt height 13pt 
{Total} & 640,149 &635,141 & 5,006 & 1,730 & $2^{1,716}\cdot 3^{14}$ & 3,276

\end{tabular}
\vspace{6pt}

\caption{\bf Escapees identified and unresolved knots remaining}
\label{EscapeesTable}
\end{table}
The results of this escapee search are summarised in Table~\ref{EscapeesTable}, which includes for each crossing range up to 21x: the number of Goeritz-bifactorizable prime alternating knots not identified as ribbon by the Algorithm; the number of these obstructed from sliceness; the number of true candidates thus remaining; the number of escapees identified, with the numbers of bands required for these; and the number of knots whose sliceness remained unresolved. A few remarks regarding the escapees and the searches for them are in order:  first, the Algorithm is not known to fail on any alternating knot requiring just one band to split to an unlink.  While this is inherent in the search outlined above, other escapee searches (all results of which are captured by the search described above) only identified escapees requiring at least two bands, as well.  Second, some escape bands on an alternating knot of $n$ crossings can result in an intermediate 2-component link for an algorithmically ribbon knot having strictly more than $n$ crossings, whose intermediate forms might not be among our 22,477 targets.  Indeed:  two escapees at 17x are identified \emph{only} due to intermediate forms resulting from algorithmically ribbon knots of 21 crossings.  Finally, \emph{non-hyperbolic} intermediate 2-component links are inherently missed during this search, so potential additional escapees would be missed should escape bands produce only intermediate forms that are non-hyperbolic; the remedy for this is a fundamentally different knot-identification method not relying solely on hyperbolic geometry.  

In sum, this resolves the sliceness of all but 3,276 of the over 1.2 billion prime alternating knots to 21 crossings.  Details including diagrams of unresolved knots are available from \url{www.klo-software.net/ribbondisks}.
As of the time of publication, no prime alternating knots of 21 crossings or fewer are known to be slice that are not identified either by the Algorithm or the escapee search; it is expected that a number of these 3,276 will be escapees requiring intermediate forms resulting from AR knots at 22 crossings or more.  Further work on identifying ribbon disks for these knots or obstructing their sliceness in some manner continues.

\clearpage

\bibliographystyle{amsplain}
\bibliography{ribbonalg}

\providecommand{\bysame}{\leavevmode\hbox to3em{\hrulefill}\thinspace}
\providecommand{\MR}{\relax\ifhmode\unskip\space\fi MR }
\providecommand{\MRhref}[2]{%
  \href{http://www.ams.org/mathscinet-getitem?mr=#1}{#2}
}
\providecommand{\href}[2]{#2}
\begin{thebibliography}{10}

\bibitem{aceto}
Paolo Aceto, \emph{Rational homology cobordisms of plumbed 3-manifolds},
  Algebr. Geom. Topol. \textbf{20} (2020), 1073--1126.

\bibitem{ammmps}
Paolo Aceto, Jeffrey Meier, Allison~N Miller, Maggie Miller, JungHwan Park, and
  Andr{\'a}s~I Stipsicz, \emph{Branched covers bounding rational homology
  ball}, arXiv:2002.10324, accepted in Algebr. Geom. Topol., 2020.

\bibitem{a2016}
S.~Akbulut, \emph{4-manifolds}, Oxford Graduate Texts in Mathematics, vol.~25,
  Oxford University Press, 2016.

\bibitem{ak}
S.~Akbulut and R.~Kirby, \emph{Branched covers of surfaces in {$4$}-manifolds},
  Math. Ann. \textbf{252} (1979/80), no.~2, 111--131.

\bibitem{bbl}
K.~L. Baker, D.~Buck, and A.~G. Lecuona, \emph{Some knots in {$S^1\times S^2$}
  with lens space surgeries}, Comm. Anal. Geom. \textbf{24} (2016), no.~3,
  431--470.

\bibitem{BL}
Kenneth~L. Baker and John Luecke, \emph{Asymmetric {L}-space knots}, Geom.
  Topol. \textbf{24} (2020), no.~5, 2287--2359.

\bibitem{brejevs}
V.~Brejevs, \emph{Ribbon surfaces for some alternating 3-braid closures},
  arXiv:2012.00577, 2020.

\bibitem{knotinfo}
J.~C. Cha and C.~Livingston, \emph{Table of knot invariants},
  \url{http://www.indiana.edu/~knotinfo}.

\bibitem{crowell}
R.~Crowell, \emph{Genus of alternating link types}, Ann. of Math. (2)
  \textbf{69} (1959), 258--275.

\bibitem{snappy}
Marc Culler, Nathan~M. Dunfield, Matthias Goerner, and Jeffrey~R. Weeks,
  \emph{Snap{P}y, a computer program for studying the geometry and topology of
  $3$-manifolds}, Available at \url{http://snappy.computop.org} (11/30/2022).

\bibitem{lconc}
A.~Donald and B.~Owens, \emph{Concordance groups of links}, Algebr. Geom.
  Topol. \textbf{12} (2012), no.~4, 2069--2093.

\bibitem{thmA}
S.~K. Donaldson, \emph{The orientation of {Y}ang-{M}ills moduli spaces and
  {$4$}-manifold topology}, J. Differential Geom. \textbf{26} (1987), no.~3,
  397--428.

\bibitem{PAKG}
O.~Flint, S.~Rankin, and P.~de~Vries, \emph{Prime alternating knot generator},
  \url{http://www-home.math.uwo.ca/~srankin/papers/knots/pakg.html}, 2003.

\bibitem{FM}
Ralph~H. Fox and John~W. Milnor, \emph{Singularities of {$2$}-spheres in
  {$4$}-space and cobordism of knots}, Osaka Math. J. \textbf{3} (1966),
  257--267.

\bibitem{goeritz}
L.~Goeritz, \emph{Knoten und quadratische {F}ormen}, Math. Z. \textbf{36}
  (1933), no.~1, 647--654.

\bibitem{GL}
C.~McA. Gordon and R.~A. Litherland, \emph{On the signature of a link}, Invent.
  Math. \textbf{47} (1978), no.~1, 53--69.

\bibitem{greene}
J.~E. Greene, \emph{Alternating links and definite surfaces}, Duke Math. J.
  \textbf{166} (2017), no.~11, 2133--2151, With an appendix by Andr{\'a}s
  Juh{\'a}sz and Marc Lackenby.

\bibitem{HKL}
Chris Herald, Paul Kirk, and Charles Livingston, \emph{Metabelian
  representations, twisted {A}lexander polynomials, knot slicing, and
  mutation}, Math. Z. \textbf{265} (2010), no.~4, 925--949.

\bibitem{howie}
J.~A. Howie, \emph{A characterisation of alternating knot exteriors}, Geom.
  Topol. \textbf{21} (2017), no.~4, 2353--2371.

\bibitem{lisca1}
P.~Lisca, \emph{Lens spaces, rational balls and the ribbon conjecture}, Geom.
  Topol. \textbf{11} (2007), 429--472.

\bibitem{lisca2}
\bysame, \emph{Sums of lens spaces bounding rational balls}, Algebr. Geom.
  Topol. \textbf{7} (2007), 2141--2164.

\bibitem{mccoy}
D.~McCoy, \emph{Alternating knots with unknotting number one}, Adv. Math.
  \textbf{305} (2017), 757--802.

\bibitem{menasco}
W.~Menasco, \emph{Closed incompressible surfaces in alternating knot and link
  complements}, Topology \textbf{23} (1984), no.~1, 37--44.

\bibitem{murasugi58}
K.~Murasugi, \emph{On the genus of the alternating knot. {I}, {II}}, J. Math.
  Soc. Japan \textbf{10} (1958), 94--105, 235--248.

\bibitem{balls}
B.~Owens, \emph{Equivariant embeddings of rational homology balls}, Q. J. Math.
  \textbf{69} (2018), no.~3, 1101--1121.

\bibitem{plesken}
W.~Plesken, \emph{Solving {$XX^{\rm tr}=A$} over the integers}, Linear Algebra
  Appl. \textbf{226/228} (1995), 331--344.

\bibitem{ras}
Jacob Rasmussen, \emph{Khovanov homology and the slice genus}, Invent. Math.
  \textbf{182} (2010), no.~2, 419--447.

\bibitem{rolfsen}
D.~Rolfsen, \emph{Knots and links}, Mathematics Lecture Series, vol.~7, Publish
  or Perish, Inc., Houston, TX, 1990, Corrected reprint of the 1976 original.

\bibitem{sartori}
A.~Sartori, \emph{Knot concordance in three manifolds}, Masters thesis,
  University of Pisa, 2010.

\bibitem{Seeliger}
A~Seeliger, \emph{Symmetrische {V}ereinigungen als {D}arstellungen von
  {B}andknoten bis 14 {K}reuzungen ({S}ymmetric union presentations for ribbon
  knots up to 14 crossings)}, Diploma thesis, Stuttgart University, 2014.

\bibitem{klo}
F.~Swenton, \emph{Knot-like objects ({KLO}) software},
  \url{http://KLO-Software.net}.

\bibitem{tait}
P.~G. Tait, \emph{On knots, {I}}, Trans. Roy. Soc. Edinburgh \textbf{28}
  (1877), 145--190.

\bibitem{sage}
{The Sage Developers}, \emph{{S}agemath, the {S}age {M}athematics {S}oftware
  {S}ystem ({V}ersion 9.3)}, 2021, {\tt https://www.sagemath.org}.

\bibitem{tsuk2}
T.~Tsukamoto, \emph{A criterion for almost alternating links to be
  non-splittable}, Math. Proc. Cambridge Philos. Soc. \textbf{137} (2004),
  no.~1, 109--133.

\bibitem{tsuk1}
\bysame, \emph{The almost alternating diagrams of the trivial knot}, J. Topol.
  \textbf{2} (2009), no.~1, 77--104.

\bibitem{turaev}
V.~G. Turaev, \emph{Classification of oriented {M}ontesinos links via spin
  structures}, Topology and geometry---{R}ohlin {S}eminar, Lecture Notes in
  Math., vol. 1346, Springer, Berlin, 1988, pp.~271--289.

\bibitem{SnapPea}
J.~Weeks, \emph{Snappea}, \url{http://www.geometrygames.org/SnapPea/}.

\end{thebibliography}

\end{document}